\documentclass{amsart}

\usepackage{amssymb,amsmath,amsthm}
\usepackage{enumerate}
\usepackage{eucal}

\usepackage{xcolor}

\title{Does weak quasi-o-minimality behave better than weak o-minimality?}

\author[S.\ Moconja]{Slavko Moconja}

\address[S.\ Moconja]{University of Belgrade\\ Faculty of Mathematics\\ Studentski trg 16, 11000 Belgrade, Serbia\and Instytut Matematyczny\\ Uniwersytet Wroc\l{}awski\\ pl.\ Grunwaldzki 2/4, 50-384 Wroc\l{}aw, Poland}

\email[S.\ Moconja]{slavko@matf.bg.ac.rs}

\thanks{Data sharing not applicable to this article as no datasets were generated or analysed during the current study.}

\thanks{The first author was supported by the Narodowe Centrum Nauki grant no.\ 2016/22/E/ST1/00450, and by the Ministry of Education, Science and Technological Development of Serbia through University of Belgrade, Faculty of Mathematics.}

\author[P.\ Tanovi\'c]{Predrag Tanovi\'c}

\address[P.\ Tanovi\'c]{Mathematical Institute SANU\\ Knez Mihailova 36, Belgrade, Serbia}

\email[P.\ Tanovi\'c]{tane@mi.sanu.ac.rs}

\thanks{The second author was supported by the Ministry of Education, Science and Technological Development of Serbia through Mathematical Institute of the Serbian Academy of Sciences and Arts.}

\newtheorem{Theorem}{Theorem}[section]
\newtheorem{Proposition}[Theorem]{Proposition}
\newtheorem{Lemma}[Theorem]{Lemma}
\newtheorem{Corollary}[Theorem]{Corollary}

\newtheorem*{Theorem*}{Theorem}
\newtheorem{TheoremI}{Theorem}

\theoremstyle{definition}
\newtheorem{Definition}[Theorem]{Definition}
\newtheorem{Fact}[Theorem]{Fact}

\newtheorem{Example}[Theorem]{Example}
\newtheorem{Remark}[Theorem]{Remark}
\newtheorem{Question}[Theorem]{Question}
\newtheorem{Claim}{Claim}

\def\endclaim{\renewcommand{\qed}{\hfill$\checkmark$}}

\def\forces{\vdash}
\def\Mon{{\mathfrak C}}
\def\B{{\mathcal B}}
\DeclareMathOperator{\tp}{tp}
\renewcommand{\leq}{\leqslant}
\renewcommand{\geq}{\geqslant}
\def\peq{\preccurlyeq}
\def\seq{\succcurlyeq}

\begin{document}

\begin{abstract}
We present a relatively simple description of binary, definable subsets of models of weakly quasi-o-minimal theories. In particular,  we closely describe definable linear orders and prove a weak version of the monotonicity theorem. We also prove that weak quasi-o-minimality of a theory with respect to one definable linear order implies weak quasi-o-minimality with respect to any other such order. 
\end{abstract}

\maketitle

\noindent
By a binary reduct of a first-order structure $(M,\ldots)$ we  mean a structure $(M,B_i)_{i\in I}$ in which  each $B_i$ is  a unary or binary definable\footnote{Throughout, definable means $0$-definable.} set. The full binary reduct is the one in which all unary and binary definable sets are named. In this paper, we are interested in those binary reducts of $\aleph_0$-saturated linearly ordered structures $ \mathcal M = (M, <, ...) $ that are simple in the sense that the family $(B_i)_{i\in I}$ consists of relatively simple geometric objects, and that (the complete theory of) the reduct eliminates quantifiers. For each such reduct, it is  interesting to determine conditions on $\mathcal M$ (or $Th(\mathcal M)$) which would guarantee that the complete binary structure of $\mathcal M$ is already definable in the reduct, i.e. whether the  reduct is interdefinable with the full binary reduct of $\mathcal M$. 
The initial motivation for our work lies in a result of Pierre Simon from \cite{Simon}.  There, he investigated the reduct 
 \begin{center}
$(M,<,C_i,R_j)_{i\in I,j\in J}$, in  which $C_i$'s name all colors (unary definable sets)\\ and $R_j$'s name all definable monotone relations (see Definition \ref{Definition monotone relation}) from $\mathcal M$, 
\end{center}
call it the cml-reduct of $\mathcal M$, and proved   elimination of quantifiers. In particular, if $\mathcal N$ is a  colored order (linear order with unary predicates), then the cml-reduct provides a relatively simple  description of all sets definable in $\mathcal N$. 
A closer description of definable sets in a colored order is offered in \cite{IMT}. There,  we introduced the ccel-reduct of $\mathcal M$:
 \begin{center}$(M,<,C_i,E_j)_{i\in I,j\in J}$, in  which $C_i$'s name all colors and 
  $E_j$'s name all 
    definable convex equivalence  relations (classes are convex),
  \end{center}
and proved that it eliminates quantifiers after naming all relations $x\leqslant S^n_{j}(y)$ and $x<S^n_{j}(y)$ ($n\in\mathbb Z$, $j\in J$), where $S^n_{j}(y)$ denotes the $n$-th successor $E_j$-class of $y$ (if such a class exists).

For a family $\mathcal F$  of definable linear orders on  $M$ with $<\in \mathcal F$,  the $\mathcal F$-monotone reduct (or simply $\mathcal F$-reduct) of $\mathcal M$ is obtained by naming all unary definable sets  and all definable   monotone relations from $(M,<)$ to $(M,\lhd)$   for some $\lhd\in\mathcal F$ (see Definition \ref{dfn monotone relation}).   
 In Proposition \ref{prop general simon}  we  provide a sufficient condition on $\mathcal F$ for the $\mathcal F$-reduct to eliminate quatifiers. In particular, each of the families  $\mathcal F=\{<\}$ and $\mathcal F=\{<,>\}$ satisfies that condition, as well as a certain family $\mathcal W_T$ (where $T=Th(\mathcal M)$) which consists of all orders $<_{\vec E}$ defined by recursively inverting the order $<$ on equivalence classes of a decreasing sequence of definable $<$-convex equivalence relations $\vec E$ (see Definition \ref{Definition <E}). 
 We will focus on $\mathcal W_T$-reducts,   or weakly monotone reducts, and introduce the notion of weak monotonicity for theories: $T$ is weakly monotone with respect to a definable linear order $<$ if for all finite $A\subseteq \Mon$ the $\mathcal W_{T_A}$-reduct  coincides (interdefinably) with the full binary reduct of $(\Mon,<,a,\ldots)_{a\in A}$; $T$ is weakly monotone if it is so with respect to some definable linear order. Surprisingly, it turns out that the class of weakly monotone theories is already known by another name. 

\begin{TheoremI}\label{theorem 1}
Let $T$ be a complete first-order theory with infinite models.
\begin{enumerate}[(i)]
\item $T$ is weakly quasi-o-minimal\footnote{See Definition \ref{dfn_WQOM}.} if and only if $T$ is weakly monotone.
\item Weak monotonicity (weak quasi-o-minimality) of $T$ does not depend on the choice of a definable linear order: if $T$ is weakly monotone with respect to one definable linear order, then it is so with respect to any other.  
\end{enumerate}
\end{TheoremI}

It is well known that the weak o-minimality of a theory depends on the choice of a definable linear order (see Example \ref{Example wom and orders}). By Theorem \ref {theorem 1}(ii) we know that this is not the case with weak quasi-o-minimality. Hence, the class of weakly quasi-o-minimal theories behaves better than the class of weakly o-minimal ones, viewed from at least one aspect. The main imperfection of that class lies in the cumbersome name; this is now fixed by Theorem \ref {theorem 1}(i).

\begin{TheoremI}\label{theorem 2}
Let $T$ be a weakly monotone theory  with a definable linear order $<$ and $\Mon$ its monster model. 
\begin{enumerate}[(i)]
 \item  Every definable subset of $\Mon^2$ is defined by a Boolean combination of unary formulae and formulae defining   weakly monotone relations.
 \item The theory of the full binary reduct of $\Mon$ eliminates quantifiers. 
 \item If $\prec$ is a definable linear order on $\Mon$, then there are a definable partition $\Mon= D_1\cup\ldots\cup D_n$   and orders  $<_{\vec E_1},\ldots,<_{\vec E_n}\in \mathcal W_T$ such that
 $\prec$ agrees with $<_{\vec E_i}$ on $D_i$ for all $i=1,\ldots,n$.
\item (Weak monotonicity theorem)
For every definable set $D\subseteq \Mon$ and definable function\footnote{See Remark \ref{Remark Dedekind completion}.} $f:D\to\overline{\Mon}$ 
there exist  a definable partition $D=D_1\cup\dots\cup D_n$ and orders $<_{\vec E_1}, \dots,<_{\vec E_n}\in\mathcal W_T$ such that each restriction $f\upharpoonright D_i$, viewed as a function from $(D_i,<_{\vec E_i})$ into $(\overline\Mon,<)$, is increasing for all $i=1,\dots,n$.
\end{enumerate}
  \end{TheoremI} 
All parts  of Theorem 2 are new even in the weakly o-minimal case. In particular, the weak monotonicity theorem can be viewed as a generalization of the local monotonicity theorem proved by Macpherson, Marker, and Steinhorn in \cite{Mac}. 
Let us also note that the simplicity of the binary structure of o-minimal theories was noticed by Mekler, Rubin, and Steinhorn in \cite{MRS}.  

\smallskip 
The paper is organized as follows.  The first section contains preliminaries and the second some basic facts on weakly quasi-o-minimal theories. In Section 3, we introduce terminology concerning monotone relations between linear orders and prove some simple facts about them. 
We start Section 4 by introducing $\mathcal F$-reducts  and isolate a condition on $\mathcal F$ implying the elimination of quantifiers in the reduct; we show that the families  $\{<\}$ and $\{<,>\}$ satisfy the condition.  Then we introduce the family $\mathcal W_T$, and a related, more technical family $\mathcal S_T$. We prove that the $\mathcal S_T$-reduct eliminates quantifiers, while only in the final part of the paper we show that the $\mathcal W_T$-reduct and the $\mathcal S_T$-reduct always coincide.  In Section 5,  we prove  the most technically demanding fact, Proposition \ref{Prop_new order on p},  in which we show that any ``new'' order on the locus of a complete type of a weakly quasi-o-minimal theory is of the form $<_{\vec E}$ for some decreasing sequence of definable, convex equivalence relations $\vec E$; as a corollary we will obtain a close description of definable total preorders in weakly 
quasi-o-minimal theories.  We will then be able to routinely complete proofs of Theorems 1 and 2 in Sections 6 and 7. 
The last section contains questions and concluding remarks.

\section*{Acknowledgment}
The present paper is a considerably modified version of the preprint named Monotone theories \cite{MoT}. It was written according to numerous comments and
suggestions from the Referee, to whom we are greatly grateful as her/his help significantly improved the quality of the paper.

\section{Preliminaries}\label{S1}

Throughout the paper we use standard model-theoretical terminology and notation.  By $L$ we denote a first-order language, by $T$ a complete $L$-theory with infinite models, and by $M$ an arbitrary model of $T$.  Unless otherwise stated, by $\Mon$ we denote a monster model of $T$. Elements of models are denoted by $a,b,\ldots$, tuples of elements by $\bar a,\bar b,\ldots$, and subsets of models (for $\Mon$, small subsets, i.e.\ subsets of cardinality strictly less than $|\Mon|$) are denoted by $A,B,\ldots$. By a formula, we will always mean an $L$-formula, unless the presence of parameters is emphasized. By an $L_A$-formula, we mean a formula with parameters from $A$. 
For a formula $\phi(\bar x)$, possibly with parameters from $A\subseteq M$, by $\phi(M)$ we denote its solution-set in $M^{|\bar x|}$. For a subset $D\subseteq M^n$ we use the following notation:

\begin{itemize}
    \item[--] $D$ is {\it $A$-definable} if $D=\phi(M)$ for some $L_A$-formula $\phi(\bar x)$ in $n$ free variables;
    \item[--] $D$ is {\it definable} if it is $\emptyset$-definable; 
    \item[--] $D$ is {\it type-definable over $A$} if it is   an intersection of $A$-definable sets;
    \item[--] If $D\subseteq S$ and $S$ is type-definable over $A$, then $D$ is {\it relatively $A$-definable in $S$} if  $D=D'\cap S$ for some $A$-definable set $D'$.
\end{itemize}

\subsection{Linear orders}
Let $(X,<)$ be a linear order. For a subset $D\subseteq X$ we use the following terminology: 
\begin{itemize}
 \item[--] $D$ is {\em convex} if $a,b\in D$ and $a<c<b$ imply $c\in D$; 
 \item[--] $D$ is an {\em initial part} if it is a downward-closed  set, i.e.\ $a\in D$ and $b<a$ imply $b\in D$. {\em Final parts} are defined dually; 
 \item[--] A subset $C\subseteq D$ is a {\it convex component} (or just a {\em component}) of $D$ if $C$ is a maximal subset of $D$ among those which are convex in $X$;
 \item[--] If $Y\subseteq X$, we will say that $D$ is {\em (relatively) convex on $Y$}, if $D\cap Y$ is convex in $(Y,<)$. The notions  $D$ is {\it (relatively) initial on $Y$} and  {\em $C\subseteq D$ is a relative convex component of $D$ in $Y$} are defined similarly. 
\end{itemize} 
For every subset $D\subseteq X$ the set of all convex components of $D$ forms a partition of $D$, so the meaning of ``$D$ has finitely many components" is clear. We will say that {\it $D$ has finitely many components on $Y\subseteq X$}  if $D\cap Y$ has finitely many components in $(Y,<)$. 
The property ``having finitely many components (on $Y$)" is clearly closed under Boolean combinations.
 
An equivalence relation $E$ on a linear order $(X,<)$ is {\em convex} if all $E$-classes are convex; if $E$ is a convex equivalence relation on $X$, then the quotient set $X/E$ is naturally linearly ordered.

All the previous definitions should be read as {\em with respect to $<$}. Although   we will often change the order that we refer to, if the meaning of $<$ is clear from the context   we will omit stressing it.

\begin{Definition} Suppose that $T$ is a complete theory with a definable linear order $<$, and let $\phi(x;\bar y)$  be a formula with the first variable separated. Fix $M\models T$. 
\begin{enumerate}[(a)] 
\item $\phi(x;\bar y)$  is {\em $<$-convex} \ if $\phi(x,\bar a)$ defines a convex subset  of $M$ for all $\bar a\in M^{|\bar y|}$;

\item $\phi(x;\bar y)$  is {\em $<$-initial} \ if $\phi(x,\bar a)$ defines an initial part  of $M$ for all $\bar a\in M^{|\bar y|}$.
\end{enumerate}
Clearly, the notions of $<$-convex and $<$-initial formulae don't depend on the choice of $M\models T$.
\end{Definition}
 When the meaning of the order $<$ is clear from the context we will usually omit mentioning it  and simply say that a formula is convex (initial).   

\begin{Remark}\label{Remark Dedekind completion}
Suppose that $<$ is a definable linear order on $\Mon$ and let  $\overline{\Mon}$ denote the corresponding completion. The concept of a sort in $\overline\Mon$ is standard, as   is the concept of a definable function from $\Mon^n$ to (a sort in) $\overline\Mon$, see subsection 1.2. in \cite{Mac}; these are certain $\Mon^{eq}$-objects, although $\overline\Mon$ itself is not a member of $\Mon^{eq}$. Each sort in $\overline\Mon$ is naturally ordered by (an extension of) $<$, as is $\overline{\Mon}$.
In this paper, we work in a single-sorted structure  $\Mon$ and use $<$-initial formulae. An alternative to that is the use of definable functions $f:\Mon^n\to\overline\Mon$: 
 any initial formula is equivalent to   $(\theta(\bar y)\land x\leqslant f(\bar y))\vee   (\lnot\theta(\bar y)\land x< f(\bar y))$ for some formula $\theta(\bar y)$ and some definable function  $f:\Mon^{|\bar y|}\to \overline\Mon$. We will use this notation only in the formulation of the weak monotonicity theorem.  
\end{Remark}

\begin{Remark}\label{remark choosing convex and initial formulae}
Let  $D_{\bar a}\subseteq Y_{\bar a}$ be $\bar a$-definable subsets of a model $M$ containing $\bar a$.

(i) If $D_{\bar a}$ is relatively convex in $Y_{\bar a}$, then  there is a convex formula $\phi(x;\bar y)$  such that  $\phi(x,\bar a)$ relatively defines $D_{\bar a}$ within $Y_{\bar a}$; for $\phi(x;\bar y)$ we may take the formula defining the convex hull (in $M$) of $D_{\bar y}\cap Y_{\bar y}$. A similar observation holds if $D_{\bar a}$ is an initial part of $Y_{\bar a}$.

(ii) If $D_{\bar a}$ has finitely many components on $Y_{\bar a}$, then each of the components is $\bar a$-definable. The leftmost one, call it $C$, is defined by: 
$$x\in D_{\bar a}\land \forall z\in D_{\bar a} (z\leqslant x\Rightarrow\forall u(z\leqslant u\leqslant x\Rightarrow u\in D_{\bar a})),$$
the next one by the same formula with $D_{\bar a}$ replaced by $D_{\bar a}\smallsetminus C$, and so on. 
\end{Remark} 

\begin{Remark}\label{remark convex fla is difference of initial}
Every $<$-convex formula $\psi(x;\bar y)$ is equivalent to $\psi_1(x,\bar y)\land\lnot\psi_2(x,\bar y)$, where $\psi_1(x;\bar y)$ and $\psi_2(x;\bar y)$ are $<$-initial formulae given by:
\begin{center}
$\psi_1(x;\bar y):= \exists x'\,(\psi(x',\bar y)\land x\leq x')\ \ \ \mbox{ and }\ \ \ \psi_2(x;\bar y):=\forall x'\,(\psi(x',\bar y)\Rightarrow x<x').$ 
\end{center}
\end{Remark}

\begin{Lemma}\label{Lema definition of fincmop sets} Suppose that $T$ is a complete theory with a definable linear order $<$ and $\Mon\supseteq A$  is its $(\aleph_0+|A|^+)$-saturated model. Let $\pi(x)$ be a partial type over $A$. 
\begin{enumerate}[(i)]
\item Suppose that $D\subseteq \Mon$  is $A$-definable. If $D$ has finitely many   convex components on $\pi(\Mon)$, then each of them is relatively defined by an $A$-instance of some convex formula. In particular, $D\cap \pi(\Mon)$ is relatively defined by a Boolean combination of $A$-instances of initial formulae.
\item Suppose that $E$ is a  relatively $A$-definable convex equivalence relation on $\pi(\Mon)$. Then there exists an $A$-definable convex equivalence relation $E_0$ which agrees with $E$ on $\pi(\Mon)\times\pi(\Mon)$.
Moreover, if $E_1(x,y)$ is an $A$-definable convex equivalence relation on $\Mon$ such that $E$ is finer than $E_1$ on $\pi(\Mon)$, then $E_0$ can be chosen finer than $E_1$.

\item If $\prec$ is a relatively $A$-definable linear order on $\pi(\Mon)$, then there is an $A$-definable linear order $\lhd$ which agrees with $\prec$ on $\pi(\Mon)$.
\end{enumerate}
\end{Lemma}
\begin{proof} (i) Suppose that $D=\phi(\Mon,\bar a)$ has finitely many  convex components on $\pi(\Mon)$;  then so does the complement $D^c=\lnot\phi(\Mon,\bar a)$. 
 Let   $n$ be the overall number of these components. Then:
$$\bigcup_{i=0}^{n}\pi(x_i)\ \cup\ \{x_i<x_{i+1}\mid i<n\}\ \vdash\ \bigvee_{i<n} (\phi(x_i, \bar a)\Leftrightarrow \phi(x_{i+1},\bar a))\,. $$
By compactness, there is  a finite conjunction $\theta(x)$ of formulae from $\pi(x)$ with:
$$ \bigwedge_{i<n}\theta(x_i)\ \land\ \bigwedge_{i<n}x_i<x_{i+1}  \vdash\ \bigvee_{i<n} (\phi(x_i, \bar a)\Leftrightarrow \phi(x_{i+1},\bar a))\,. $$
This means that the overall number $N$ of  convex components of $D$ and $D^c$ within $\theta(\Mon)$ is at most $n$. On the other hand,  $\pi(\Mon)\subseteq \theta(\Mon)$ implies  $n\leqslant N$. Thus, $n=N$ and for each component $C_i$ within $\theta(\Mon)$,  $\pi(\Mon)\subseteq \theta(\Mon)$ implies that $C_i\cap\pi(\Mon)$ is a component within $\pi(\Mon)$.  $C_i$ is $A$-definable  by Remark \ref{remark choosing convex and initial formulae}(ii) and by part (i) of that remark the defining formula may be chosen to be an $A$-instance of a convex formula; let $C$ be the union of all $C_i$'s that are 
components of $D$ within $\theta(\Mon)$.  Clearly, $C$ is a Boolean combination of $A$-instances of convex (and hence of initial by Remark \ref{remark convex fla is difference of initial}) formulae and $x\in C$ relatively defines $D\cap \pi(\Mon)$.

\smallskip
(ii) Let $\varepsilon(x,y)$ be an $L_A$-formula relatively defining $E$ on $\pi(\Mon)$. By compactness we can find a finite conjunction of formulae from $\pi(x)$, $\theta(x)$, such that $\varepsilon(x,y)$ defines a convex equivalence relation on $\theta(\Mon)$ (finer than $E_1 $ there). Then the relation  $E_0\subseteq\Mon\times \Mon$   defined by the following formula works:
$$(\exists\, u,v\in\theta(\Mon))(u\leq x,y\leq v\wedge\varepsilon(u,v))\ \lor\ x=y.$$
 
\smallskip
(iii) Again by compactness we can find a finite conjunction of formulae from $\pi(x)$, $\theta(x)$, such that $x\prec y$ defines a linear order on $\theta(\Mon)$. Now we define $\lhd$ such that $\theta(\Mon)\lhd\neg\theta(\Mon)$, $\lhd$ agrees with $\prec$ on $\theta(\Mon)$, and with $<$ on $\neg\theta(\Mon)$.
\end{proof}

\subsection{Defining new orders}

Let $(X,<)$ be a linear order  and  $E$ a convex equivalence relation on $X$. Define:
$$x<_Ey\ \mbox{ iff }\ (E(x,y)\mbox{ and }y<x)\mbox{ or }(\neg E(x,y)\mbox{ and }x<y).$$
It is easy to see that $<_E$ is a linear order on $X$; it reverses the order within each $E$-class, but the classes in the factor order remain ordered originally. In particular, $E$ is a $<_E$-convex equivalence relation and the factor orders induced by $<$ and $<_E$ on $X/E$ are equal.
We can further iterate the construction.

\begin{Definition}\label{Definition <E} Let $(X,<)$ be a linear order and 
  $\vec E=(E_1,E_2,\ldots,E_n)$ a sequence of equivalence relations on $X$.
  \begin{enumerate}[(a)]
  \item We say that the sequence $\vec E$ is {\it $<$-admissible} if $E_1$ is $<$-convex, $E_2$ is $<_{E_1}$-convex, $E_3$ is $(<_{E_{1}})_{E_2}$-convex, etc. 
  \item If $\vec E$ is an admissible sequence, then  \  $<_{\vec E}:=(\ldots((<_{E_1})_{E_2})_{E_3}\ldots)_{E_{n}}$. 
  \end{enumerate}
\end{Definition}

Usually,   $(X,<)$ and  $\vec E$ will be  $A$-definable, in which case  the resulting order $<_{\vec E}$ will be  $A$-definable, too. 

\begin{Remark}\label{remark decreasing is admissible}
(i) For any $<$-convex equivalence relation $E$ we have $<_{(E,E)}=<$.

(ii) Note that for any $<$-convex equivalence relation $E'$ which is either coarser or finer than  a $<$-convex equivalence relation $E$ we have that $E'$ is $<_E$-convex; in particular, $(E,E')$ is an admissible sequence. 

(iii) Any decreasing sequence $\vec E=(E_1,\ldots,E_n)$ of $<$-convex equivalence relations ($E_{i+1}$ refines $E_i$ for all relevant $i$) is admissible: by induction one proves that $E_{m}$ is $<_{E_1,\ldots,E_k}$-convex for all $k<m\leqslant n$, then part (ii) applies and one can continue the induction.  
\end{Remark}

\begin{Lemma}\label{remark admissible reversed order}
Let $\vec E=(E_1,\ldots,E_{n})$ be a $<$-admissible sequence on $(X,<)$.
\begin{enumerate}[(i)]
    \item Sequences $\vec E_1:=(X\times X,E_1,\ldots,E_{n})$ and $\vec E_2:=(E_1,\ldots,E_{n}, X\times X)$ are also $<$-admissible and $<_{\vec E_1}=<_{\vec{E_2}}$ is just the reversed order of  $<_{\vec E}$.
    \item The sequence $\vec E^*=(E_n,\ldots,E_1)$,  is $<_{\vec E}$-admissible and  $(<_{\vec E})_{\vec E^*}=<$\,.
\end{enumerate}
\end{Lemma}
\begin{proof}
(i) is easy. To prove (ii) notice that  $E_n$ is $<_{(E_1,\ldots, E_{n-1})}$-convex  since $\vec E$ is $<$-admissible. Hence $(E_n,E_n)$ is $<_{(E_1,\ldots, E_{n-1})}$-admissible, so  \begin{center}$(<_{\vec E})_{E_n}=(<_{(E_1,\ldots, E_{n-1})})_{(E_n,E_n)}=<_{(E_1,\ldots, E_{n-1})}$.\end{center} We conclude by induction.
\end{proof}

\begin{Lemma}\label{remark convex set and new orders} 
Let $\vec E=(E_1,\ldots,E_n)$ be a $<$-admissible sequence on $(X,<)$ and let $D\subseteq X$. Then $D$ has finitely many $<$-convex components if and only if it has finitely many $<_{\vec E}$-convex components. 
\end{Lemma}
\begin{proof}
First we show that any $<$-convex set $D\subseteq X$ has finitely many $<_{\vec E}$-convex components. 
If $n=1$, then $D$ intersects properly at most two $E_1$-classes (the endpoints of $D/E_1$), hence $D$ has at most three $<_{E_1}$-convex components. Each of these components has at most three $(<_{E_1})_{E_2}$-components, and so on; $D$ has at most $3^n$ $<_{\vec E}$-components. It follows that a subset with $m$ $<$-convex components has at most $m\cdot 3^n$\, $<_{\vec E}$-convex components. 

Assume that $D'\subseteq X$ has $m$\, $<_{\vec E}$-convex components and let $\vec E^*=(E_n,\ldots,E_1)$. Then  $(<_{\vec E})_{\vec E^*}=<$, so by the above (applied to $D'$, $\vec E^*$, and $<_{\vec E}$)
$D'$ has at most $m\cdot 3^n$\, $<$-convex components. \end{proof}

\section{Basic facts on weakly quasi-o-minimal theories}\label{section_basic_wqom}
Weakly quasi-o-minimal theories were introduced by Kuda\u\i bergenov in \cite{Kud} as a generalization of both weakly o-minimal and quasi-o-minimal theories; the latter were introduced by Belagradek, Peterzil and Wagner  in \cite{BPW}.
Originally, Kuda\u\i ber\-genov's definition   assumes that a linear order is a part of the language. Here we  require only that $T$ has a definable linear order, as it will turn out that weak quasi-o-minimality does not depend on the choice of the order.

\begin{Definition}\label{dfn_WQOM} 
\begin{enumerate}[(a)]
    \item A theory $T$ is {\em weakly quasi-o-minimal with respect to a definable  linear order $<$} if for some $\aleph_0$-saturated model $\Mon\models T$ every parametrically definable subset of $\Mon$ is a finite Boolean combination of $<$-convex sets and (unary $L$-)definable sets; 
    \item A theory $T$ is {\em weakly quasi-o-minimal} if it is weakly quasi-o-minimal with respect to some definable linear order $<$. 
\end{enumerate}
\end{Definition}

Note that in the first part of the definition we ask that the desired condition holds in some $\aleph_0$-saturated model. However, as we shall see in Proposition \ref{prop_WQOM_characterization}(ii), if the desired condition holds in some $\aleph_0$-saturated model, it holds in every model. Also, note that in the definition we did not require the convex sets to be definable.

\begin{Remark}\label{rmk naming parameters and qwom}
Clearly,  weak  quasi-o-minimality of the theory is preserved under naming parameters. However, it is not preserved in reducts: consider the theory $T$ of the structure $(\mathbb Q,<,E)$, where $E$ is an equivalence relation with two topologically dense classes. Note that none of the classes is a Boolean combination of definable and convex sets. Hence $T$ is not weakly quasi-o-minimal, but naming a single parameter makes it such. 
\end{Remark}

\begin{Proposition}\label{prop_WQOM_characterization} 
Let $T$ be a complete theory with a definable linear order $<$.  
\begin{enumerate}[(i)]
\item Let $\Mon\models T$ be $\aleph_0$-saturated. The following conditions are  equivalent:
\begin{enumerate}[(1)]
\item $\Mon$ witnesses that $T$ is weakly quasi-o-minimal with respect to $<$.

\item For all $p\in S_1(T)$ every $\Mon$-definable set $D\subseteq\Mon$ has finitely many $<$-convex components on $p(\Mon)$.

\item Every unary $L_{\bar a}$-formula is a Boolean combination of unary $L$-formulae and $\bar a$-instances of $<$-convex $L$-formulae (for all $\bar a\subseteq\Mon$). 
\end{enumerate}
\item If $T$ is weakly quasi-o-minimal with respect to $<$, then every parametrically definable subset of every model is a finite Boolean combination of $<$-convex sets and definable sets. In particular, every $\aleph_0$-saturated model witnesses that $T$ is weakly quasi-o-minimal with respect to $<$.
\end{enumerate}
\end{Proposition}
\begin{proof} 
(i) (1)$\implies$(2): Assume that $\Mon$ witnesses that $T$ is weakly quasi-o-minimal with respect to $<$ and that $D\subseteq\Mon$ is $A$-definable.   Since both the family of unary definable sets and the family of finite unions of convex sets are closed under Boolean combinations, we may write $D$ as $\bigcup_{i\leqslant n}U_i\cap C_i$, where  $U_i$'s form a definable partition of $\Mon$   and each $C_i$ is a finite union of convex sets. If $p\in S_1(T)$, then $p(\Mon)\subseteq U_i$ holds for some $i\leqslant n$, whence $p(\Mon)\cap D=p(\Mon)\cap C_i$; since $C_i$  has finitely many convex components on $p(\Mon)$, so does $D$. 

\smallskip
(2)$\implies$(3): Suppose that condition (2) holds. We will prove that a formula  $\phi(x,\bar a)$  is equivalent to a Boolean combination of unary $L$-formulae and ${\bar a}$-instances of convex $L$-formulae. Fix $p\in S_1(T)$.  Condition  (2) implies that the set 
$D=\phi(\Mon,\bar a)$ has finitely many   convex components within  $p(\Mon)$, so by Lemma \ref{Lema definition of fincmop sets}(i) there is  a Boolean combination of convex $L$-formulae $\phi_p(x;\bar y)$ with $\phi_p(x,\bar a)$ relatively defining $D\cap p(\Mon)$: \ 
 $p(x) \vdash \phi_p(x,\bar a)\Leftrightarrow \phi(x,\bar a)$. \
By compactness, there is a formula $\theta_p(x)\in p(x)$ such that:
 $\models \theta_p(x) \Rightarrow (\phi_p(x,\bar a)\Leftrightarrow \phi(x,\bar a))$. \  In other words, the set $D_{p}=D\cap \theta_p(\Mon)$ is definable by a Boolean combination of 
 unary $L$-formulae and ${\bar a}$-instances of convex $L$-formulae. 
Since $\{[\theta_p]\mid p\in S_1(T)\}$ covers $S_1(T)$,  it has a  finite subcover $\{[\theta_{p_i}]\mid i\leqslant k\}$. Then $D=\bigcup_{i\leqslant k} D_{p_i}$ is definable by a  Boolean combination of  
 unary $L$-formulae  and ${\bar a}$-instances of convex $L$-formulae. This completes the proof of $(2)\implies(3)$. Since  $(3)\implies(1)$ trivially holds the proof of (i) is complete.
 
\smallskip
(ii) Let $\Mon$ be an $\aleph_0$-saturated model witnessing that $T$ is weakly quasi-o-minimal with respect to $<$, let $M\models T$, and let $D\subseteq M$ be a subset defined by $\phi(x,\bar m)$, $\bar m\subseteq M$. Take a realization $\bar a$ of $\tp(\bar m)$ in $\Mon$. By (3) of (i), $\phi(x,\bar a)$ is equivalent to a Boolean combination of $L$-formulae and $\bar a$-instances of $<$-convex $L$-formulae; denote this Boolean combination by $B(x,\bar a)$. Thus $\forall x(\phi(x,\bar y)\Leftrightarrow B(x,\bar y))\in \tp(\bar a)=\tp(\bar m)$, so $\phi(x,\bar m)$ is equivalent to $B(x,\bar m)$ in $M$. Since each formula participating in $B(x,\bar y)$ is either an unary $L$-formula or a $<$-convex $L$-formula, we conclude that $D$ is a Boolean combination of definable and $<$-convex sets.
\end{proof}

For the rest of the section we work in a monster model $\Mon$ of $T$.

\begin{Corollary}\label{cor wqom fla is ui expressible} Suppose that $T$ is weakly quasi-o-minimal with respect to $<$. Then every binary formula   $\phi(x,y)$ is equivalent to a Boolean combination of unary and binary $<$-initial formulae of the form $\psi(x;y)$.
\end{Corollary} 
\begin{proof}Fix $\phi(x,y)$, $p\in S_1(T)$, and  $b\models p$. By Proposition \ref{prop_WQOM_characterization}(i)(3), the formula $\phi(x,b)$ is equivalent to a Boolean combination of unary formulae and $\{b\}$-instances of convex formulae. By Remark \ref{remark convex fla is difference of initial} each convex formula may be replaced by a Boolean combination of initial formulae; denote the obtained
formula by $\theta_p(x,b)$. Clearly, $p(y)\forces \phi(x,y)\Leftrightarrow\theta_p(x,y)$. By compactness there is $\sigma_p(y)\in p(y)$ such that $\sigma_p(y)\forces \phi(x,y)\Leftrightarrow \theta_p(x,y)$. The set $\{[\sigma_p]\mid p\in S_1(T)\}$ covers $S_1(T)$ so by compactness we can extract a finite subcover $\{[\sigma_{p_i}]\mid 1\leq i\leq n\}$. Now $\phi(x,y)$ is clearly equivalent to the disjunction $\bigvee_{i=1}^n(\sigma_{p_i}(y)\land\theta_{p_i}(x,y))$,  which is a desired Boolean combination.
\end{proof}

\begin{Lemma}\label{Plem_change_of_order_by_equivalence} If $T$ is weakly quasi-o-minimal with respect to $<$ and $\vec E$ is an admissible sequence of  definable  equivalence relations, then $T$ is weakly quasi-o-minimal with respect to $<_{\vec E}$\,, too.
\end{Lemma}
\begin{proof} We will prove that $T$ is weakly quasi-o-minimal with respect to $<_E$  for any definable $<$-convex equivalence relation $E$; the general case follows by induction. Let $D\subseteq \Mon$ be $\Mon$-definable and let $p\in S_1(T)$. By Proposition \ref{prop_WQOM_characterization}(i)(2) it suffices to show that $D$ has finitely many $<_E$-convex components on $p(\Mon)$, knowing that it has finitely many $<$-convex components there; that follows by Lemma \ref{remark convex set and new orders}.
\end{proof}

\begin{Lemma}\label{Lema_equivalences in qwom}
If $T$ is weakly quasi-o-minimal with respect to $<$, then any definable equivalence relation, when restricted to the locus of $p\in S_1(T)$, is  convex.
\end{Lemma}
\begin{proof}
  Assume that for some $p\in S_1(T)$ some $E$-class is not $<$-convex on $p(\Mon)$. We can find $a_0,b_0,a_1\models p$ such that $a_0<b_0<a_1$, $E(a_0,a_1)$ and $\neg E(a_0,b_0)$. Consider $f\in Aut(\Mon)$ with $f(a_0)=a_1$, and define $a_{n+1}=f(a_n)$ and $b_{n+1}=f(b_n)$ for $n\geq 0$. By induction we clearly have $a_n<b_n<a_{n+1}$, $E(a_0,a_n)$, and $\neg E(a_0,b_n)$. This means that the class $E(a_0,\Mon)$ has infinitely many convex components on $p(\Mon)$, which is contrary to Proposition \ref{prop_WQOM_characterization}(i)(2).
\end{proof}

\begin{Example}\label{Example wom and orders}
Consider the structure $\mathcal M=(\mathbb R,<,\mathbb Q)$ and define $\lhd$ by: 
$$\lhd:= (<\cap\ (\mathbb I\times \mathbb I\cup \mathbb Q\times \mathbb Q))\cup\mathbb (\mathbb I\times \mathbb Q),$$
where $\mathbb I:=\mathbb R\smallsetminus\mathbb Q$ is the set of irrational numbers. The relation $\lhd$ is a (definable) linear order on $\mathbb R$, $\lhd$ agrees with $<$ on both $\mathbb I$ and $\mathbb Q$,  and   $\mathbb I\lhd \mathbb Q$. By standard arguments $Th(\mathcal M)$ eliminates quantifiers and $\mathcal M$ is $\aleph_0$-saturated, so one can easily see that $\mathcal M$ (and $Th(\mathcal M)$) is weakly o-minimal with respect to $\lhd$, but it is not so with respect to $<$. 
\end{Example}

\section{Monotone relations}

In this section we work with arbitrary linear orders.

\begin{Definition}\label{Definition monotone relation}
A binary relation $R\subseteq A\times B$ between linear orders  $(A,<_A)$ and $(B,<_B)$ is {\it monotone-increasing} (or simply {\it monotone}) if:
\begin{center}
$a\leqslant_A a'$, \ $a'\,R\,b'$ \ and \ $b'\leqslant_B b$ \ \ imply \ \ $a\,R\,b$. 
\end{center}
\end{Definition}
Binary monotone relations were introduced by Simon in \cite{Simon} and used in our recent paper \cite{MT2}.
Basic examples of monotone relations $R\subseteq A\times B$ are those defined by $f(x)\leqslant_B y$
 or  $f(x)<_B y$ for some increasing function $f:A\to B$. More generally, if $(C,<_C)$ is a linear order, $g:A\to C$ and $h:B\to C$ increasing functions, then $g(x)<_C h(y)$ and $g(x)\leqslant_C h(y)$ also define monotone relations.  

\smallskip
For a binary relation $R\subseteq A\times B$, $a\in A$ and $b\in B$, by $R(a,B)$  we will denote the fiber $\{y\in B\mid a\,R\,y\}$; similarly for $R(A,b)$. The following fact is easily verified and will be used without mention.

\begin{Fact}\label{Fact_monotonicity}
Let 
$(A,<_A)$ and $(B,<_B)$ be linear orders and $R\subseteq A\times B$. Then $R$   is a monotone relation if and only if  any of the following, mutually  equivalent conditions holds:
\begin{enumerate}[(1)]
\item  $(R(A,b)\mid b\in B)$  is a $\subseteq$-increasing sequence of initial parts  of $A$, i.e.\ $R(A,b)$ is an initial part of $(A,<_A)$ for each $b\in B$, and $b_1<_Bb_2$ implies $R(A,b_1)\subseteq R(A,b_2)$;  
\item  $(R(a,B)\mid a\in A)$ is a $\subseteq$-decreasing  sequence  of final parts  of $B$.
\end{enumerate}
\end{Fact}

Since we will consider monotone relations on a fixed domain with various orders, the following definition is appropriate.

\begin{Definition}\label{dfn monotone relation} Let $(X,<)$ and $(X,\prec)$ be linear orders.  A relation $R\subseteq X\times X$ is {\em $(<,\prec)$-monotone} if $R$ is a monotone relation from $(X,<)$ to $(X,\prec)$. 
\end{Definition}

Note that $<$  itself is trivially a $(<,<)$-monotone relation, while $(<,>)$-monotone relations may be called monotone-decreasing.

\begin{Lemma}\label{lemma monotonicity basic} Suppose that $R,R'\subseteq X\times X$ are respectively $(<,\prec)$- and $(<',\prec')$-monotone relations, where $<,\prec,<',\prec'$ are linear orders on $X$. Then:
\begin{enumerate}[(i)]
\item The complement $R^c$ is $(>,\succ)$-monotone and the inverse $R^{-1}$ is a $(\succ,>)$-monotone relation.
$(R^{-1})^c$ is a $(\prec,<)$-monotone relation.

\item By $R(X,x)\subseteq R'(X,y)$ is defined a $(\prec,\prec')$-monotone relation.
\item If $D\subseteq X$, then by $R(X,x)\cap R'(X,y)\subseteq D$ is defined a $(\prec,\succ')$-monotone relation.
\end{enumerate}
\end{Lemma}
\begin{proof} 
(i) By $(<,\prec)$-monotonicity of $R$, $a\leq a'\,R\,b'\peq b$ implies $a\,R\,b$, so $a'\geq a\,R^c\,b\seq b'$ implies $a'\,R^c\,b'$. Also, $b\seq b'\,R^{-1}\,a'\geq a$ implies $a\leq a'\,R\,b'\peq b$, hence $a\,R\,b$, i.e.\ $b\,R^{-1}\,a$, follows by $(<,\prec)$-monotonicity of $R$. The last part follows from the previous two.

\smallskip
(ii) Assume $a\peq a'$, $R(X,a')\subseteq R'(X,b')$ and $b'\peq ' b$. For $t\in R(X,a)$ we have $t\,R\,a\peq a'$, so $t\,R\,a'$ by $(<,\prec)$-monotonicity of $R$. Since $R(X,a')\subseteq R'(X,b')$ we get $t\,R'\,b'$, which together with $b'\peq' b$ gives $t\,R'\,b$ by $(<',\prec')$-monotonicity of $R'$. Thus $t\in R'(X,b)$ and $R(X,a)\subseteq R'(X,b)$ follows.

\smallskip
(iii) Assume $a\peq a'$, $R(X,a')\cap R'(X,b')\subseteq D$ and $b'\seq' b$. For $t\in R(X,a)\cap R'(X,b)$ we have $t\,R\,a\peq a'$ and $t\,R'\,b\peq' b'$, so $t\,R\,a'$ and $t\,R'\,b'$ follow by appropriate monotonicities. Thus $t\in R(X,a')\cap R'(X,b')$ and hence $t\in D$. Therefore, $R(X,a)\cap R'(X,b)\subseteq D$.
\end{proof}

\section{$\mathcal F$-monotone reducts}

In this section we work with an arbitrary complete theory $T$ with a definable linear order $<$. By $\mathcal F$ we denote a set of definable linear orders on $(\Mon,<,\dots)$ with $<\in \mathcal F$.

\begin{Definition}\begin{enumerate}[(a)]
\item A binary relation on $\Mon$ is {\it $\mathcal F^<$-monotone} if it is $(<,\lhd)$-monotone for some $\lhd\in\mathcal F$.

\item A formula $\phi(x,y)$ is  {\it $\mathcal F^<$-monotone} if it defines an $\mathcal F^<$-monotone relation. 
Let us stress that by saying "$\mathcal F^<$-monotone formula of the form $\phi(x,y)$" we mean that $\phi$ should be understood as a monotone relation from $(\Mon_x,<)$ to $(\Mon_y,\lhd)$ (and not from $(\Mon_y,<)$ to $(\Mon_x,\lhd)$) for some $\lhd\in\mathcal F$. 

\item The {\it $\mathcal F^<$-monotone reduct} of $(\Mon,<,\dots)$, or simply the $\mathcal F^<$-reduct, is the  structure $(\Mon,<, C_i,R_j)_{i\in I,j\in J}$ obtained by naming  all unary definable sets  by  $C_i$'s and all definable $\mathcal F^<$-monotone relations by $R_j$'s.
\end{enumerate}
\end{Definition}
 
\begin{Remark}
(i) Whenever the meaning of the order $<$ is clear from the context, we will omit the superscript and say that a formula is $\mathcal F$-monotone; similarly for $\mathcal F$-monotone relations and $\mathcal F$-reducts. 

(ii) The $\mathcal F$-reduct   depends on the choice of the language (and interpretation) but,  up to interdefinability, it is uniquely determined; we will assume that the language and the interpretation are in some (unspecified) way canonically chosen and denote by $T_{\mathcal F}$ the complete  theory of the reduct. 

(iii) Whenever we say that two reducts of the same structure coincide, we will mean that they are interdefinable. 

(iv) For $\mathcal F=\{<\}$  the $\mathcal F$-reduct is the cml-reduct.
\end{Remark}

There are two natural  questions to ask at this point:

\begin{enumerate}[{  \rm  Question} 1.]
\item When does $T_{\mathcal F}$ eliminate quantifiers?
\smallskip \item When does the $\mathcal F$-reduct coincide with the full binary reduct?
\end{enumerate}

\begin{Definition}  We say that the set $\mathcal F$ is {\it $T$-complete} if for all $\lhd,\lhd'\in\mathcal F$ every $(\lhd,\lhd')$-monotone definable relation $R(x,y)$  is defined by a Boolean combination of unary formulae and $\mathcal F$-monotone formulae $\phi(x,y)$.
\end{Definition}  
We can now partially answer Question 1. 
\begin{Proposition}\label{prop general simon}  If $\mathcal F$ is  $T$-complete, then  $T_{\mathcal F}$  eliminates quantifiers.
\end{Proposition}
\begin{proof}
We will adapt Simon's argument from the proof of Proposition 4.1 in \cite{Simon}. We operate in the $\mathcal F$-reduct $(\Mon,<,C_i,R_j)_{i\in I,j\in J}$ and  first we claim that  for all
$a_0$ and  $\bar a=(a_1,\ldots,a_n)$ the quantifier-free type 
$\mathrm{tp}_\mathrm{qf}(a_0/\bar a)$ is determined by:
\begin{itemize}
\item all colors $C_i(x)$, where $C_i(x)\in\tp(a_0)$, and 
\item all $R_j(x,a_{k})$ and $\neg R_j(x,a_{k})$   belonging to $\tp(a_0/\bar a)$.
\end{itemize}
Denote this type by $\pi_{\bar a}(x)$. It suffices to prove that each atomic formula from $\tp(a_0/\bar a)$ is implied by $\pi_{\bar a}(x)$. Each such a formula without parameters is (modulo $T_{\mathcal F}$) already in $\pi_{\bar a}(x)$ as $\{C_i\mid i\in I\}$ names all unary definable sets.
Note that $<$,  as  a  $(<,<)$-monotone relation,  is one of the $R_j$'s.
Hence it remains to show that whether $R_j(a_k,x)$ belongs to $\tp(a_0/\bar a)$ or not, is decided by $\pi_{\bar a}(x)$.  Fix such $j,k$ and 
suppose that $\phi(y,x):=R_j(y,x)$ is $(<,\lhd)$-monotone. Then the relation $(R_j^{-1})^c$ (defined by $\psi(x,y):=\lnot R_j(y,x)$) is $(\lhd,<)$-monotone by Lemma \ref{lemma monotonicity basic}(i). By $T$-completeness of $\mathcal F$, $\lnot R_j(y,x)$  is $T$-equivalent to a Boolean combination of unary and $\mathcal F$-monotone formulae $R_{s}(x,y)$, and this information is contained in $T_{\mathcal F}$. In particular, whether $R_j(a_k,x)$ or $\lnot R_j(a_k,x)$ belongs to $\tp(a_0/\bar a)$ is decided and  $\pi_{\bar a}(x)\cup T_{\mathcal F}\cup\mathrm{Diag}(\bar a)\vdash \mathrm{tp}_\mathrm{qf}(a_0/\bar a)$, where $\mathrm{Diag}(\bar a)$ denotes the quantifier-free diagram of $\bar a$ in the $\mathcal F$-reduct. This proves the claim.   

To prove the proposition, it is enough that for given $a_0$, $\bar a=(a_1,\ldots,a_n)$ and $\bar b=(b_1,\ldots,b_n)$ satisfying  $\bar a\equiv_\mathrm{qf}\bar b$  we find $b_0$ with $a_0\bar a\equiv_\mathrm{qf}b_0\bar b$; in other words, we should show that the set $\pi_{\bar b}(x)$, obtained by substituting $\bar b$ in place of $\bar a$ in $\pi_{\bar a}(x)$, is consistent.  
By compactness, it suffices to verify consistency of  finite subsets; fix a finite $\pi_0(x)\subseteq \pi_{\bar b}(x)$ consisting of:
\begin{itemize}
\item  colors $C_i(x)$ for $i\in I_0$;
\item  formulae  $R_{j}(x,b_{k})$ for  $(j,k)\in J^0\subset J\times \{1,\dots,n\}$;
\item  formulae  $\neg R_{j}(x,b_{k})$ for  $(j,k)\in J^1\subset J\times \{1,\dots,n\}$. 
\end{itemize}
Let $C(x)=\bigwedge_{i\in I_0} C_i(x)$. Further, consider   $(R_j(\Mon,a_k)\mid (j,k)\in J^0)$; by Fact \ref{Fact_monotonicity} it consists  of initial parts of $(\Mon,<)$, so let $R_{j_0}(\Mon,a_{k_0})$ be its minimal element. Then for every $(j,k)\in J^0$ we have $R_{j_0}(\Mon,a_{k_{0}})\subseteq R_{j}(\Mon,a_{k})$. 
By Lemma \ref{lemma monotonicity basic}(ii), $R_{j_0}(\Mon,x)\subseteq R_{j}(\Mon,y)$ defines a $(\lhd_{j_0},\lhd_{j})$-monotone relation (where $R_{j_0}$ is $(<,\lhd_{j_0})$-monotone and $R_j$ is $(<,\lhd_j)$-monotone). By $T$-completeness of $\mathcal F$, that relation is defined by a Boolean combination of  unary formulae and  those defining $\mathcal F$-monotone relations, so the information $R_{j_0}(\Mon,a_{k_{0}})\subseteq R_{j}(\Mon,a_{k})$ is contained in $\mathrm{tp}_\mathrm{qf}(\bar a)$ and hence in $\mathrm{tp}_\mathrm{qf}(\bar b)$, too. Hence $R_{j_0}(\Mon,b_{k_0})$ is the minimal element of $(R_j(\Mon,b_k)\mid (j,k)\in J^0)$. 
Similarly, we find a minimal element $\neg R_{j_1}(\Mon,b_{k_1})$   of $(\neg R_j(\Mon,b_k)\mid (j,k)\in J^1)$, corresponding to the minimal element $\lnot R_{j_1}(\Mon,a_{k_1})$ among final parts in  $(\lnot R_j(\Mon,a_k)\mid (j,k)\in J^1)$.

It remains  to find $b_0\in C(\Mon)\cap R_{j_0}(\Mon,b_{k_{0}})\cap \neg R_{j_1}(\Mon,b_{k_1})$. By Lemma \ref{lemma monotonicity basic}(iii), $R_{j_0}(\Mon,x)\cap \neg R_{j_1}(\Mon,y)\subseteq\lnot C(\Mon)$ defines a
$(\lhd_{j_0},\lhd_{j_1})$-monotone relation. 
  Since $a_0$ witnesses $R_{j_0}(\Mon,a_{k_{0}})\cap\neg R_{j_1}(\Mon,a_{k_1})\nsubseteq\lnot C(\Mon)$, and by $T$-completeness of $\mathcal F$ this is determined by $\mathrm{tp}_\mathrm{qf}(\bar a)$, we have $R_{j_0}(\Mon,b_{k_{0}})\cap \neg R_{j_1}(\Mon,b_{k_1})\nsubseteq \lnot C(\Mon)$, hence the desired $b_0$ exists.
\end{proof}

Clearly, the set $\mathcal F=\{<\}$  is   $T$-complete, so Simon's result on elimination of quantifiers in the cml-reduct is a special case of Proposition \ref{prop general simon}.
Yet another special case is when $\mathcal F=\{<,>\}$; that this one is $T$-complete easily follows by  Lemma \ref{lemma monotonicity basic}(i).

\begin{Corollary}\label{Corollary vece manje}  The set  $\mathcal F=\{<,>\}$ is $T$-complete and $T_{\mathcal F}$ eliminates quantifiers. 
\end{Corollary}

We  now turn to the case when $\mathcal F$ contains only relations of the form $<_{\vec E}$. Denote: 
\begin{itemize}
\item $\mathcal S_T^<$ -- the set of all orders of the form $<_{\vec E}$ for some admissible sequence of  definable equivalence relations $\vec E$;
\item $\mathcal W_T^<$ -- the set of all  $<_{\vec E}\in  \mathcal S_T^<$ for  $\vec E$ decreasing (recall Remark \ref{remark decreasing is admissible}(iii));
\item  A $\mathcal W_T^<$-monotone formulae will be called {\it weakly   monotone with respect to $<$},   and a $\mathcal S_T^<$-monotone formulae {\it special monotone with respect to $<$}. 
\end{itemize}
Whenever the meaning of $<$ is clear from the context, we will simply write  $\mathcal S_T$ and $\mathcal W_T$, or  say that a formula is weakly monotone (special monotone). 

Clearly, every weakly monotone formula is special monotone. Our main interest is in weakly monotone formulae  since $<_{\vec E}$ is easier to visualise  when  $\vec E$ decreases than  in the general  admissible case. Special formulae are introduced only to approximate weakly monotone ones and in the final part of the paper we will prove that the $\mathcal S_T$-reduct and the $\mathcal W_T$-reduct coincide, and that $\mathcal W_T$ is $T$-complete. However, the proof of $T$-completeness of  $\mathcal S_T$ is technically much simpler than that for $\mathcal W_T$.  
 
\begin{Proposition}\label{prop general via special} The set  $\mathcal S_T$ is $T$-complete. \end{Proposition}
\begin{proof}
We will  show that every definable  $(<_{\vec E_1},<_{\vec E_2})$-monotone relation, where $\vec E_1$ and $\vec E_2$ are definable  admissible sequences, is defined by a Boolean combination of unary and special monotone formulae. The proof is an easy induction on $|\vec E_1|$, once we prove the following claim.

\smallskip
\textsc{Claim.} Suppose that $\prec$ and $\lhd$ are  definable linear orders, $E$ is a  definable $\prec$-convex equivalence relation and $R$ is a definable $(\prec_E,\lhd)$-monotone relation. Then there exist  a definable $\lhd$-convex equivalence relation $F$,  $(\prec,\lhd)$-monotone relations $R_1$ and $R_2$, and a  $(\prec,\lhd_F)$-monotone relation $R_3$ such that $R(x,y)$ is equivalent to  $R_1(x,y)\lor (R_2(x,y)\land  \lnot R_3(x,y))$.

To prove the claim first note that by Fact \ref{Fact_monotonicity} the $\prec_E$-initial parts  $(R(\Mon,y)\mid y\in \Mon)$\,   $\subseteq$-increase   when $y$'s $\lhd$-increase. 
Denote by $c(y)$ the union of all $E$-classes meeting $R(\Mon,y)$, so $c(y)$'s are $\prec$-initial parts which $\subseteq$-increase when $y$'s $\lhd$-increase. 
Being a $\prec_E$-initial part,   $R(\Mon,y)$ consists of at most two $\prec$-convex components. Denote by $c_1(y)$ the one that consists of all $E$-classes that are completely contained in $R(\Mon,y)$ and the other by $c_2(y)=R(\Mon,y)\smallsetminus c_1(y)$; here we allow that either of them is void.
Like $c(y)$'s, $c_1(y)$'s are $\prec$-initial parts which $\subseteq$-increase when $y$'s $\lhd$-increase, as well. 
If $c_2(y)\neq\emptyset$ then $c(y)\smallsetminus c_1(y)$ is a single $E$-class (the one  properly met by $R(\Mon,y)$). Define:
$$R_1(x,y):=x\in c_1(y), \ \ \ R_2(x,y):=x\in c(y), \ \ \ R_3(x,y):=x\in c(y)\smallsetminus c_2(y).$$
Clearly, $\models  R(x,y)\Leftrightarrow (R_1(x,y)\lor (R_2(x,y)\land  \lnot R_3(x,y)))$. Since $c(y)$ and $c_1(y)$ are $\prec$-initial parts that $\subseteq$-increase when $y$'s $\lhd$-increase, $R_1$ and $R_2$ are  $(\prec,\lhd)$-monotone relations. 

Let $F(y,z)$ be defined by $c(y)=c(z)\land c_1(y)=c_1(z)$. Clearly, $F$ is an equivalence relation and it is easy to see that it is $\lhd$-convex:  $y\lhd t\lhd z$ and $F(y,z)$, by monotonicity of $R_1$ and $R_2$, imply $F(z,t)$. 
To complete the proof, it remains to show that $R_3$ is $(\prec,\lhd_F)$-monotone. Clearly, $R_3(\Mon,y)=c(y)\smallsetminus c_2(y)$ is a $\prec$-initial part. Assume  $y\lhd_Fz$ and we will prove  $R_3(\Mon,y)=c(y)\smallsetminus c_2(y)\subseteq c(z)\smallsetminus c_2(z)=R_3(\Mon,z)$.

\smallskip
Case 1. \ $(y,z)\in F$. In this case  we have  $c_1(z)=c_1(y)\subseteq c(y)=c(z)$ and $z\lhd y$. The latter, by monotonicity of $R$, implies $R(\Mon,z)\subseteq R(\Mon,y)$. Hence  $c_2(y)=R(\Mon,y)\smallsetminus c_1(y)\supseteq R(\Mon,z)\smallsetminus c_1(z)=c_2(z)$ and 
$c(y)\smallsetminus c_2(y)\subseteq c(z)\smallsetminus c_2(z)$.

\smallskip
Case 2. \ $(y,z)\notin F$. In this case  $y\lhd_F z$ implies $y\lhd z$, hence $R(\Mon,y)\subseteq R(\Mon,z)$, $c_1(y)\subseteq c_1(z)$ and $c(y)\subseteq c(z)$. 
We claim that $c(y)\subseteq c_1(z)$.  Otherwise, since these are initial parts, we have $c_1(z)\subsetneq c(y)$, so
 $c_1(y)\subseteq c_1(z)\subsetneq c(y)\subseteq c(z)$.  
Since  $c(t)\smallsetminus c_1(t)$ is either empty or is a single $E$-class and $c(t)$ and $c_1(t)$ are $E$-closed subsets, we conclude $c_1(y)=c_1(z)\subsetneq c(y)=c(z)$ and hence $(y,z)\in F$. A contradiction. 
Thus
 $c(y)\subseteq c_1(z)$, hence   $c(y)\smallsetminus c_2(y)\subseteq c(y)\subseteq c_1(z)\subseteq  c(z)\smallsetminus c_2(z)$ and we are done.

\smallskip
In both cases we reached the desired conclusion, completing the proof of the claim and the proposition. 
\end{proof}

\begin{Corollary}\label{Cor S_T reduct is wqom}
Let $T$ be a complete theory with a definable linear order $<$. The theory $T_{\mathcal S_T}$ is weakly quasi-o-minimal with respect to $<$ and eliminates quantifiers. 
\end{Corollary}
\begin{proof}
By Propositions \ref{prop general simon} and \ref{prop general via special} the theory $T_{\mathcal S_T}$ eliminates quantifiers. To show that $T_{\mathcal S_T}$ is weakly quasi-o-minimal with respect to $<$, let $D$ be a  parametrically $T_{\mathcal S_T}$-definable subset of $\Mon$ and $p\in S_1(T_{\mathcal S_T})$. 
Elimination of quantifiers implies that $D$ is a Boolean combination of unary, $<$-initial and $<_{\vec E}$-final sets; each of them has finitely many $<$-convex components on $p(\Mon)$ (recall that by Lemma \ref{remark convex set and new orders}, $<_{\vec E}$-convex sets have finitely many $<$-components). Therefore, $D$ has finitely many $<$-convex components on $p(\Mon)$, too. Thus condition  (2) from Proposition \ref{prop_WQOM_characterization}(i) is satisfied and the theory $T_{\mathcal S_T}$ is weakly quasi o-minimal with respect to $<$.
\end{proof}

\section{Definable orders in weakly quasi o-minimal theories}\label{S3}

In this section, we go back to the study of  weakly quasi-o-minimal theories begun in Section 2. The main result is Theorem \ref{thm wqom definable order} in which we closely describe total preorders definable on (elements of) the monster model $(\Mon,<,...)$ of a weakly quasi-o-minimal theory. However, most of our efforts are aimed at proving Proposition \ref{Prop_new order on p} in which we show that the only ``new'' definable linear orders on the locus of a complete 1-type are of the form $<_{\vec E}$.

\begin{Lemma}\label{Lema semiintervals}  
Suppose that $T$ is weakly quasi-o-minimal with respect to $<$, $p=\tp(a)$ and that $I(a)$ is a relatively $\{a\}$-definable $<$-initial part of $p(\Mon)$.

\begin{enumerate}[(i)]
\item If $a\in I(a)$, then there do not exist $a_0,b_0,a_1\models p$ such that $a_1<b_0<a_0$, $I(a_1)\supseteq I(a_0)$, and $I(b_0)<a_0$.

\item If $a\notin I(a)$, then there do not exist $a_0,b_0,a_1\models p$ such that $a_0<b_0<a_1$, $I(a_1)\subseteq I(a_0)$, and $a_0\in I(b_0)$.
\end{enumerate}
\end{Lemma} 
\begin{proof}We will prove only part (i); part (ii) is proved similarly. 
Assume that $I(a)$ is relatively defined by $\phi(x,a)$. Then $a\in I(a)$ implies $\models\phi(a,a)$ for all $a\models p$. Toward a contradiction assume that $a_0,b_0,a_1\models p$ are such that $a_1<b_0<a_0$, $I(a_1)\supseteq I(a_0)$ and $I(b_0)<a_0$. Thus: $$\phi(\Mon,a_1)\cap p(\Mon)\supseteq \phi(\Mon,a_0)\cap p(\Mon)\ \ \ \mbox{ and }\ \ \ \phi(\Mon,b_0)\cap p(\Mon)<a_0.$$ 
By compactness there is a formula $\theta(x)\in p(x)$ such that:
$$\phi(\Mon,a_1)\cap \theta(\Mon)\supseteq \phi(\Mon,a_0)\cap \theta(\Mon)\ \ \ \mbox{ and }\ \ \ \phi(\Mon,b_0)\cap \theta(\Mon)<a_0.$$
By changing $\phi(x,y)$ with $\phi(x,y)\land\theta(x)$ we get: $\phi(x,a)$ relatively defines the initial part $I(a)$ of $p(\Mon)$ for all $a\models p$, $\models\phi(a,a)$ holds for all $a\models p$,   $\phi(\Mon,a_1)\supseteq\phi(\Mon,a_0)$, and $\phi(\Mon,b_0)<a_0$.

Choose $f\in Aut(\Mon)$ with $f(a_0)=a_1$, and define sequences $(a_n)_{n<\omega}$ and $(b_n)_{n<\omega}$ of realizations of $p$ by $a_{n+1}=f(a_n)$ and $b_{n+1}=f(b_n)$. By induction we see that $a_{n+1}<b_n<a_n$, $\phi(\Mon,a_{n+1})\supseteq \phi(\Mon,a_n)$, and $\phi(\Mon,b_n)<a_n$ hold for all $n<\omega$. 
Then  for all $n<\omega$ $\models\phi(a_0,a_n)\land \lnot\phi(a_0,b_n)$, so $a_0>b_0>a_1>b_1>a_2>\ldots$ alternately satisfy the formula $\phi(a_0,y)$ and its negation on $p(\Mon)$;   $\phi(a_0,\Mon)$ has infinitely many $<$-convex components on $p(\Mon)$, which contradicts Proposition \ref{prop_WQOM_characterization}(i)(2).
\end{proof}
The reader may notice that the proof of the previous lemma uses only a local version of the weak quasi-o-minimality:  every relatively $L_\Mon$-definable subset of $p(\Mon)$ has finitely many $<$-convex components within $p(\Mon)$. The same observation holds for the next proof, too.

\begin{Proposition}\label{Prop_new order on p} Suppose that $T$ is weakly quasi-o-minimal with respect to $<$.
If $\lhd$ is a  definable linear order on $\Mon$ and $p\in S_1(T)$, then there is     $<_{\vec E}\in \mathcal W_T$ agreeing with $\lhd$  on $p(\Mon)$. 
\end{Proposition}
\begin{proof} We will assume that $p$ is a non-algebraic type, as the statement is trivial for algebraic types.
First, we introduce a notation. Assume that $\prec$ is a definable linear order such that $T$ is weakly quasi-o-minimal with respect to $\prec$ (e.g.\ $\prec$ is $<$ or $<_{\vec E}$; see Lemma \ref{Plem_change_of_order_by_equivalence}).
Let $a$ realize $p$. Consider the subsets of $p(\Mon)$ relatively defined by $a\lhd x$ and $x\lhd a$. 
By Proposition \ref{prop_WQOM_characterization}(i), each of them has finitely many $\prec$-convex components on $p(\Mon)$  and all these components together with $\{a\}$ form a finite $\prec$-convex partition of $p(\Mon)$; denote this partition by $\mathcal C^\prec_a$ and note that $|\mathcal C^\prec_a|\geqslant 3$.
The partition $\mathcal C_a^\prec$ is ordered by $\prec$, and for every two $\prec$-consecutive members of the partition different from $\{a\}$, it holds that elements of one of them satisfy $a\lhd x$ while the elements of the other satisfy $x\lhd a$. Define:
\begin{itemize} 
\item $L_a^\prec$ is the $\prec$-leftmost and $R_a^\prec$ is the $\prec$-rightmost element of $\mathcal C_a^\prec$;
\item $\B^\prec_a$ is $\mathcal C_a^\prec\smallsetminus \{L_a^\prec,R_a^\prec\}$ and $B_a^\prec=\bigcup\B_a^\prec$; 
\item $l_a^\prec$ is the $\prec$-leftmost and $r_a^\prec$ is the $\prec$-rightmost element of $\B_a^\prec$.
\end{itemize}
By Lemma \ref{Lema definition of fincmop sets}(i),  each component from $\mathcal C^\prec_a$, as well as the union $B^\prec_a$, are relatively $a$-definable.
It is easy to see that $B_a^\prec$ is a  $\prec$-convex subset of $p(\Mon)$ containing $a$, $\B_a^\prec$ is a $\prec$-convex partition of $B_a^\prec$, while $L_a^\prec\prec B_a^\prec\prec R_a^\prec$ is a $\prec$-convex partition of $p(\Mon)$ (in particular, $L_a^\prec\prec a\prec R_a^\prec$ holds). 

The proof can be sketched as follows: we prove that $y\in B_x^\prec$ defines an equivalence relation on $p(\Mon)$, call it $E$. If $E$ is nontrivial (i.e.\ $B^\prec_a\neq\{a\}$) we prove $\mathcal B^{\prec_E}_a=\mathcal B^{\prec}_a\smallsetminus \{l_a^\prec ,r_a^\prec \}$. Starting with $\prec=<$ and continuing in this way, we obtain a decreasing sequence $\vec E$ and reach the case when $E$ is trivial; in that case $\lhd$ is either $<_{\vec E}$ or $>_{\vec E}$ and we are done.

\begin{Claim}
The following conditions are equivalent: \  \ (1) $|\mathcal C_a^\prec|=3$; \ (2) \ $r_a^{\prec}=\{a\}$; \ (3) \ $l_a^{\prec}=\{a\}$; \ (4) \ $B_a^{\prec}=\{a\}$; \ (5) \ Either $\prec$ or $\succ$ agrees with $\lhd$ on $p(\Mon)$. 
\end{Claim}
\begin{proof}[Proof of Claim 1]\endclaim
(5)$\Rightarrow$(1), (1)$\Rightarrow$(4), (4)$\Rightarrow$(2) and (4)$\Rightarrow$(3) are obvious.

(2)$\Rightarrow$(5)  If $r_a^{\prec}=\{a\}$, then $R_a^{\prec}=(a,+\infty)_\prec\, \cap\, p(\Mon)$. We have two cases to consider. The first is when $R^{\prec}_a$ is a component of $a\lhd x$, in which case $a\prec x\Rightarrow a\lhd x$ holds on $p(\Mon)$; since $\lhd$ and $\prec$ are linear orders and $a\in p(\Mon)$ is arbitrary, we conclude that $\lhd$ and $\prec$ agree on $p(\Mon)$. Similarly, in the second case, when $R^\prec_a$ is a component of $x\lhd a$, we get that $\lhd$ and $\succ$ agree on $p(\Mon)$.

(3)$\Rightarrow$(5) is proved similarly.  
\end{proof}

\begin{Claim}
Either $\forall x\in L_a^{\prec}\cup R_a^{\prec}(a\lhd x\Leftrightarrow a\prec x)$, \ or   \ $\forall x\in L_a^{\prec}\cup R_a^{\prec}(a\lhd x\Leftrightarrow x\prec a)$.
\end{Claim}
\begin{proof}[Proof of Claim 2]\endclaim
First, we find $a'$ with $B_a^{\prec}\prec B_{a'}^{\prec}$: Choose $b\in L_a^{\prec}$ and $f\in Aut(\Mon)$ with $B_a^\prec \prec f(b)$. Then $b\prec B_a^\prec\prec f(b)$ implies $b\prec B_a^{\prec}\prec f(b)\prec B_{f(a)}^{\prec}$,  so $a'=f(a)$ works. Note that $B_a^{\prec}\prec B_{a'}^{\prec}$ implies $a'\in R_{a}^{\prec}$ and $a\in L_{a'}^{\prec}$. 

We have two cases to consider and the first is when $a\lhd a'$ holds. In this case, $a'\in R_{a}^{\prec}$ and $a\lhd a'$ imply that $R_a^{\prec}$ is a $\prec$-convex component of $a\lhd x$; hence,  $\forall x\in R_a^{\prec}(a\lhd x\Leftrightarrow a\prec x)$ holds.  Similarly, $L_{a'}^{\prec}$ is a $\prec$-convex component of $x\lhd a'$, so $L_a^{\prec}$ is a $\prec$-convex component of $x\lhd a$ and we have $\forall x\in L_a^{\prec}(a\lhd x\Leftrightarrow a\prec x)$. Therefore, in the first case all elements of $L_a^{\prec}\cup R_a^{\prec}$ satisfy $a\lhd x\Leftrightarrow a\prec x$. The proof in the second case is similar: $a'\lhd a$ implies that all elements of $L_a^{\prec}\cup R_a^{\prec}$ satisfy   $a\lhd x\Leftrightarrow a\succ x$. 
\end{proof}

Our next task is proving that $y\in B_x^{\prec}$ defines an equivalence relation on $p(\Mon)$. That we will do assuming that the first option given by Claim 2 holds:
\begin{center}
\hfill all elements of $L_a^{\prec}\cup R_a^{\prec}$ satisfy \ $a\lhd x\Leftrightarrow a\prec x$.\hfill($\dagger$)
\end{center}
If the second option holds for $\prec$, then the first one holds for $\succ$, so $\mathcal C^{\prec}_x=\mathcal C^{\succ}_x$ and  $B^{\prec}_x=B^{\succ}_x$ imply the desired conclusion in that case, too. So till the end of proof of Claim 5, assume ($\dagger$). Also,  assume $B_a^{\prec}\neq \{a\}$; otherwise, $y\in B_x^{\prec}$ is a trivial equivalence relation.

 By Claim 1 we have $l_a^{\prec}\neq\{a\}\neq r_a^{\prec}$, so $L_a^{\prec}\prec l_a^{\prec}\prec a\prec r_a^{\prec}\prec R_a^{\prec}$. 
Condition ($\dagger$) implies that elements of $L_a^{\prec}$ satisfy $x\lhd a$, while elements of  $R_a^{\prec}$ satisfy $a\lhd x$. Since 
the sets $L_a^{\prec}\prec l_a^{\prec}$ are $\prec$-consecutive members of the partition $\mathcal C_a^{\prec}$, elements of $l_a^{\prec}$ must satisfy $a\lhd x$; similarly, elements of $r_a^{\prec}$ satisfy $x\lhd a$. 
  
In claims below,   $x,y,\ldots$  will denote realizations of $p$. 

\begin{Claim} $x\preccurlyeq y$ implies  $L_x^{\prec}\subseteq L_y^{\prec}$ and $R_x^{\prec}\supseteq R_y^{\prec}$.
\end{Claim}
\begin{proof}[Proof of Claim 3.]\endclaim
We first prove that $x\prec y$ implies $R_x^{\prec}\supseteq R_y^{\prec}$. Otherwise, since  $R_x^{\prec},\,R_y^{\prec}$ are final parts of $p(\Mon)$, there are $x\prec y$ with $R_x^{\prec}\subset R_y^{\prec}$; then $(R_y^{\prec})^c\subset (R_x^{\prec})^c$, where the complements are taken in $p(\Mon)$. Since $(R_y^{\prec})^c$, $(R_x^{\prec})^c$ are initial parts of $p(\Mon)$  and $r_x^{\prec}$ is a final part of $(R_x^{\prec})^c$, there is $z\in r_x^{\prec}$ with $(R_y^{\prec})^c\prec z$.
We {\em claim}  that $R_x^{\prec}\subseteq R_z^{\prec}$. Otherwise,   $R_z^{\prec}\subset R_x^{\prec}$ implies $(R_x^{\prec})^c\subset (R_z^{\prec})^c$ and, arguing as above, we find $t\in r_z^{\prec}$ with $(R_x^{\prec})^c\prec t$; in particular, $t\in R_x^{\prec}$. By ($\dagger$) we have: $t\in R_x^{\prec}$ implies $x\lhd t$, $t\in r_z^{\prec}$ implies $t\lhd z$, and $z\in r_x^{\prec}$ implies $z\lhd x$; a contradiction! Hence $R_x^{\prec}\subseteq R_z^{\prec}$, i.e. $(R_z^{\prec})^c\subseteq (R_x^{\prec})^c$.

 Consider the initial part $I(x)=(R_x^{\prec})^c$ of $p(\Mon)$; since $R_x^{\prec}$ is relatively definable, so is $I(x)$. We have $x\prec y\prec z$, $I(x)\supseteq I(z)$, and $I(y)\prec z$, which contradict  Lemma \ref{Lema semiintervals}(i) and prove  the $R^{\prec}$-part of the claim. 
The $L^{\prec}$-part  is proved analogously, except that for   a contradiction one should use Lemma \ref{Lema semiintervals}(ii)  with $I(x)=L_x^{\prec}$.
\end{proof}

\begin{Claim} 
$y\in B_x^{\prec}$ implies $B_y^{\prec}\subseteq B_x^{\prec}$. 
\end{Claim}
\begin{proof}[Proof of Claim 4.]\endclaim
Assume that $y\in B_x^{\prec}$ and $x\prec y$. Note that  $B_y^{\prec}\subseteq B_x^{\prec}$ is equivalent with the conjunction of $L_x^{\prec}\subseteq L_y^{\prec}$ and $R_x^{\prec}\subseteq R_y^{\prec}$. By Claim 3, $x\prec y$ implies $L_x^{\prec}\subseteq L_y^{\prec}$, so it remains to prove   $R_x^{\prec} \subseteq R_y^{\prec}$. Toward contradiction assume $R_y^{\prec} \subset R_x^{\prec}$.

As $y\in B_x^{\prec}$, there is $z\in r_x^{\prec}$ with $y\preccurlyeq z$. By Claim 3 we have $R_z^{\prec}\subseteq R_y^{\prec}$, which combined with $R_y^{\prec}\subset  R_x^{\prec}$ implies $R_z^{\prec}\subset R_x^{\prec}$, so there is $t\in r_z^{\prec}$ with $B_x^{\prec}\prec t$, i.e.\ $t\in R_x^{\prec}$. 
By ($\dagger$) we have: $t\in R_x^{\prec}$ implies $x\lhd t$,  $t\in r_z^{\prec}$ implies $t\lhd z$, and $z\in r_x^{\prec}$ implies $z\lhd x$; a contradiction! This proves the claim in the case $x\prec y$; the proof in the other case is analoguos. 
\end{proof}

\begin{Claim} 
$y\in B_x^{\prec}$ is an equivalence relation on $p(\Mon)$. \end{Claim}
\begin{proof}[Proof of Claim 5.]\endclaim
Toward contradiction, assume  $y\in B_x^{\prec}$ and $B_x^{\prec}\neq B_y^{\prec}$. By Claim 4 we have   $B_y^{\prec}\subseteq B_x^{\prec}$, so $B_x^{\prec}\not\subseteq B_y^{\prec}$ which, by Claim 4 again, implies $x\notin B_y^{\prec}$ i.e $x\in L_y^{\prec}\cup R_y^{\prec}$. We will assume 
that $x\in L_y^{\prec}$; the proof in the case   $x\in R_y^{\prec}$ is similar. 
$B_y^{\prec}\subseteq B_x^{\prec}$ implies that there is $z\in r_x^{\prec}$ with $y\preccurlyeq z$. 
By Claim 3 we have $L_y^{\prec}\subseteq L_z^{\prec}$, so $x\in L_y^{\prec}$ implies $x\in L_z^{\prec}$. By ($\dagger$) we have:  $x\in L_z^{\prec}$ implies $x\lhd z$ and  $z\in r_x^{\prec}$ implies $z\lhd x$; a contradiction!
\end{proof}
We have just proved that $y\in B^\prec_x$ is an  equivalence relation on $p(\Mon)$, so we drop the additional assumptions  ($\dagger$) and $B_a^{\prec}\neq\{a\}$. Since the relation $y\in B^\prec_x$ is relatively definable and its classes $B_x^\prec$ are $\prec$-convex on $p(\Mon)$, by Lemma 
\ref{Lema definition of fincmop sets}(ii) there is a definable $\prec$-convex equivalence relation $E$ relatively defining $y\in B^\prec_x$ on $p(\Mon)$. 

\begin{Claim} If $E$ is nontrivial ($B_a^{\prec}\neq\{a\}$), then \ 
$L_a^{\prec_E}=L_a^\prec\cup r_a^\prec$, \  $R_a^{\prec_E}=R_a^\prec\cup l_a^\prec$, \ and \ $\mathcal B_a^{\prec_E}=\mathcal B^\prec_a\smallsetminus \{l_a^\prec,r_a^\prec\}$.
\end{Claim}
\begin{proof}[Proof of Claim 6.]\endclaim
Consider $\prec_E$ and recall that it reverses the order $\prec$ within each $E$-class and keeps the order of $E$-classes. It is not hard to see that all the elements $L_a^{\prec}\prec l_a^{\prec}\prec\ldots\prec  a\prec\ldots\prec r_a^{\prec}\prec R_a^{\prec}$ of  the partition  $\mathcal C_a^\prec$  
remain $\prec_E$-convex with $L_a^{\prec}\prec_E r_a^{\prec}\prec_E\ldots\prec_E a\prec_E\ldots\prec_E l_a^{\prec}\prec_E R_a^{\prec}$; here $\prec_E$ reverses the $\prec$-order on elements of $\B_a^\prec$ because they all lie in the same $E$-class. 

By Claim 2 $L_a^{\prec}$ and $R_a^{\prec}$ are  $\prec$-components of different ($x\lhd a$ or $a\lhd x$) kinds. The same holds for $R_a^{\prec}$ and $r_a^{\prec}$, since they are consecutive.
Hence, $L_a^{\prec}$ and $r_a^{\prec}$ are  $\prec$-components of the same kind. Since they are $\prec_E$-consecutive, $L_a^{\prec}\cup r_a^{\prec}$ is a single $\prec_E$-component. Similarly, $R_a^{\prec}\cup l_a^{\prec}$ is a $\prec_E$-component. 
\end{proof}
 
 Now, we can finish the proof of the proposition.
 Construct a sequence $\prec_0,\prec_1,\ldots$ of definable linear orders and a decreasing sequence $E_1,E_2,\ldots$ of definable $<$-convex equivalence relations in the following way:  Start with $\prec_0:=<$ and let  $E_1$ be a definable $<$-convex equivalence relation relatively defining $y\in B_x^{\prec_{0}}$ on $p(\Mon)$.  If $E_1$ is trivial on $p(\Mon)$, then by Claim 1 $\lhd$ agrees with either $<$ or $>=<_{\Mon\times\Mon}$ on $p(\Mon)$; in either case we are done.  So assume that $E_1$ is nontrivial and let $\prec_1:=(\prec_0)_{E_1}$. By Claim 6 we have  $\B_x^{\prec_1}\subset \B_x^{\prec_0}$, so the equivalence relation defined by $y\in \B_x^{\prec_{1}}$ refines $E_1$ on $p(\Mon)$. By Lemma \ref{Lema definition of fincmop sets}(ii), there is a definable, convex equivalence relation $E_2$ which refines $E_1$ and relatively defines $y\in \B_x^{\prec_{1}}$ within $p(\Mon)$. If $E_2$ is trivial on $p(\Mon)$, then by Claim 1 $\lhd$ agrees with either $<_{E_1}$ or its reverse. The latter order is $(<_{E_1})_{\Mon\times \Mon}=<_{(\Mon\times\Mon,E_1)}$; in either case we are done. Assuming that $E_2$ is nontrivial we continue by defining $\prec_2:=(\prec_1)_{E_2}$\ldots. By Claim 6 at each step of the construction the size of $\B_x^{\prec_k}$ decreases by 2, so the construction has to stop after 
 $n=(|\B_x^<|-1)/2$ steps, i.e. $E_n$ is trivial on $p(\Mon)$. Then   $\lhd$ agrees with either $<_{(E_1,\ldots, E_{n-1})}$ or  $<_{(\Mon\times\Mon, E_1,\ldots, E_{n-1})}$ on $p(\Mon)$; in either case we are done.
\end{proof}

For a total preorder $(X,\peq)$, let $E_\peq$ be an equivalence relation defined  by $a\peq b\land b\peq a$; then   $\peq/E_\peq$    linearly orders  $X/E_\peq$.

\begin{Corollary}\label{cor total preorder on a type} Suppose that $T$ is weakly quasi-o-minimal with respect to $<$ and that $\peq$ is a definable total preorder. Then for each type $p\in S_1(T)$ there exist a definable convex equivalence relation $F$ agreeing with $E_\peq$ on $p(\Mon)$, and an order $<_{\vec E}\in \mathcal W_T$ such that on $p(\Mon)$: 
\begin{center}
$x\peq y$ \ \ \ is equivalent with \ \ \ $F(x,y)\lor (\lnot F(x,y)\land x<_{\vec E}y)$.
\end{center}
\end{Corollary}
\begin{proof}
Let $R(x,y):= (E_\peq(x,y)\land x<y)\lor (\lnot E_\peq(x,y)\land x\peq y)$.  \ 
By Lemma \ref{Lema_equivalences in qwom}, for any $p\in S_1(T)$, each $E_\peq$-class is $<$-convex on $p(\Mon)$, thus we can easily see that $R$ relatively defines a linear order on $p(\Mon)$. By Lemma \ref{Lema definition of fincmop sets}(iii) there exists a definable linear order $\lhd$  agreeing with $R$ on $p(\Mon)$. By Proposition \ref{Prop_new order on p}   $\lhd$ agrees with some $<_{\vec E}\in\mathcal W_T$ on $p(\Mon)$. Clearly, $x\peq y$ and $E_\peq(x,y)\lor(\lnot E_\peq(x,y)\land x<_{\vec E}y)$ are equivalent on $p(\Mon)$. By Lemma \ref{Lema_equivalences in qwom} the relation $E_\peq$ is convex on $p(\Mon)$, so by Lemma \ref{Lema definition of fincmop sets}(ii) there exists a definable, convex equivalence relation $F$ agreeing with $E_\peq$ on $p(\Mon)$. Hence $x\peq y$ and $F(x,y)\lor(\lnot F(x,y)\land x<_{\vec E}y)$ are equivalent on $p(\Mon)$.
\end{proof}

\begin{Theorem}\label{thm wqom definable order} Suppose that $T$ is weakly quasi-o-minimal with respect to $<$ and that $\peq$ is a definable total preorder on $\Mon$. Then there are a definable partition $\Mon= D_1\cup\ldots\cup D_n$ and orders  $<_{\vec E_1},\ldots,<_{\vec E_n}\in \mathcal W_T$ such that
\begin{center}
$x\peq y$ \ \ \ is equivalent with \ \ \ $E_\peq(x,y)\lor (\lnot E_\peq(x,y)\land x<_{\vec E_i}y)$
\end{center}
holds on $D_i$, while the restriction of $E_\peq$ is convex on $D_i$  for all $i=1,\ldots,n$.

Moreover, if $\peq$ is a total order, then the partition and the orders can be chosen so that $\prec$ and $<_{\vec E_i}$ agree on each $D_i$.
\end{Theorem}
\begin{proof}
By Corollary \ref{cor total preorder on a type}, for every $p\in S_1(T)$ there exist a definable convex equivalence relation $F_p$ agreeing with $E_\peq$ on $p(\Mon)$ and a finite, decreasing sequence $\vec E_p$ of definable $<$-convex equivalence relations such that $\peq$ agrees with $F_p(x,y)\lor(\lnot F_p(x,y)\land x<_{\vec E_p}y)$ on $p(\Mon)$. By compactness the same holds on $\theta_p(\Mon)$ for some $\theta_p\in p$. Since $\{[\theta_p]\mid p\in S_1(T)\}$ covers $S_1(T)$, by compactness we can extract a finite subcover $\{[\theta_{p_1}],\ldots,[\theta_{p_n}]\}$. By taking $D_k$ to be the set defined by $\theta_{p_k}\land\bigwedge_{i<k}\neg\theta_{p_i}$, for all $k$, we obtain a finite definable partition of $\Mon$. Since $\peq$ and $F_{p_k}(x,y)\lor(\lnot F_{p_k}(x,y)\land x<_{\vec E_{p_k}}y)$ agree on $\theta_{p_k}(\Mon)$, they also agree on $D_k$. This completes the proof of the first part of the theorem as $F_{p_k}$ agrees with $E_{\peq}$ on $\theta_{p_k}(\Mon)$ (and hence on $D_k$).
The second part is an easy consequence of the first one because $E_\peq$ is the equality for total orders $\peq$.
\end{proof}
 
\section{Proof of Theorem 1}

\begin{Definition} 
\begin{enumerate}[(a)]
\item A theory $T$ with a definable linear order $<$ is {\em almost weakly monotone (almost special monotone) with respect to $<$} if its full binary reduct coincides with the monotone $\mathcal W_T$-reduct (monotone $\mathcal S_T$-reduct).
\item An $\aleph_0$-saturated structure $(\Mon,\ldots)$ with a definable linear order $<$ is {\em weakly monotone (special monotone) with respect to $<$} if for all finite $A\subseteq\Mon$ the theory $T_A$ is almost weakly monotone (almost special monotone) with respect to $<$.
\item A theory $T$ with a definable linear order $<$ is {\em weakly monotone (special monotone) with respect to $<$} if it has a weakly monotone (special monotone) $\aleph_0$-saturated  model.
\end{enumerate}
\end{Definition}

\begin{Lemma}\label{cor initial formula is um definabilna} Suppose that $T$ is weakly quasi-o-minimal with respect to $<$. Then every initial formula $\phi(x;y)$ is equivalent to a Boolean combination of unary and weakly monotone formulae of the form $\psi(x,y)$.
\end{Lemma}
\begin{proof}
Let $\phi(x;y)$ be an initial formula. Define $y_1\peq y_2$ by $\phi(\Mon,y_1)\subseteq \phi(\Mon,y_2)$. Since  initial parts are $\subseteq$-comparable, $\peq$ is a total preorder on $\Mon$. By Theorem \ref{thm wqom definable order} there is a definable partition $\Mon=D_1\cup\ldots\cup D_n$ 
and orders  $<_{\vec E_1},\ldots,<_{\vec E_n}\in \mathcal W_T$ such that   $y_1\peq y_2$ is equivalent to $E_{\peq}(y_1,y_2)\lor (\lnot E_\peq(y_1,y_2)\land y_1<_{\vec E_i}y_2)$ on $D_i$ for every $i=1,\ldots,n$.
We {\it claim}  that $\models (\phi(x,y)\land y\in D_i)\Leftrightarrow(\psi_i(x,y)\land y\in D_i)$ holds for some   weakly monotone formula $\psi_i(x,y)$ defining a $(<,<_{\vec E_i})$-monotone  relation.
To prove it, we first show that $\phi(x,y)\land y\in D_i$ defines a monotone relation from $(\Mon,<)$ to $(D_i,<_{\vec E_i})$:  for $b_1,b_2\in D_i$, $b_1<_{\vec E_i}b_2$ clearly implies $b_1\peq b_2$ i.e.\ $\phi(\Mon,b_1)\subseteq \phi(\Mon,b_2)$, so $(\phi(\Mon,b_i)\mid b_i\in D_i)$ is a $\subseteq$-increasing sequence of initial parts of $\Mon$. Hence $\phi(x,y)\land y\in D_i$ defines a monotone relation from $(\Mon,<)$ to $(D_i,<_{\vec E_i})$. 
Extend it  to a  $(<,<_{\vec E_i})$-monotone relation defined by: 
$$\psi_i(x,y):= \exists u,v(\phi(u,v)\land v\in D_i\land x\leqslant u\land v\leqslant_{\vec E_i}y).$$
Clearly, the formula $\psi_i(x,y)$ is weakly monotone and satisfies the claim. It follows that 
\ $\models \phi(x,y)\Leftrightarrow \bigvee_{i=1}^n(\psi_i(x,y)\land y\in D_i)$.
\end{proof}

\begin{Proposition}\label{thm mon iff wqom} Let $T$ be a complete theory with a definable linear order $<$. The following conditions are equivalent:

\begin{enumerate}[(1)]
\item $T$ is weakly quasi-o-minimal with respect to $<$;
\item For all $A$, each $L_A$-formula $\phi(x,y)$ is equivalent to a Boolean combination of unary and weakly monotone $L_A$-formulae of the form $\psi(x,y)$;
\item $T$ is weakly monotone with respect to $<$;
\item $T$ is special monotone with respect to $<$.
\end{enumerate}

\end{Proposition}
\begin{proof}
(1)$\implies$(2) Let $T$ be weakly quasi-o-minimal with respect to $<$. Fix $A\subseteq\Mon$ and an $L_A$-formula $\phi(x,y)$. By Remark \ref{rmk naming parameters and qwom} the theory $T_A$ is weakly quasi-o-minimal with respect to $<$, so by 
Corollary \ref{cor wqom fla is ui expressible} (applied to $T_A$), $\phi(x,y)$ is equivalent to a Boolean combination of unary and  initial $L_A$-formulae of the form $\theta(x;y)$. By Lemma \ref{cor initial formula is um definabilna} (applied to $T_A$), each $\theta(x,y)$ is equivalent  to a Boolean combination of unary and  weakly monotone $L_A$-formulae of the form $\psi(x,y)$; so is  $\phi(x,y)$.

\smallskip
(2)$\implies$(3) is easy and for (3)$\implies$(4)  it suffices to note that binary sets definable in the $\mathcal W_{T_A}$-reduct are definable in the $\mathcal S_{T_A}$-reduct, too.

\smallskip
(4)$\implies$(1): Suppose that $T$ is special monotone with respect to $<$. We will verify condition  (2) from Proposition \ref{prop_WQOM_characterization}(i): By induction on $|\bar a|=n$ we prove that every $\bar a$-definable set has finitely many convex components on the locus of any $p\in S_1(T)$. The case $n=0$ is clear, so assume that it holds for $n=|\bar a|$. It suffices to  prove that $\phi(\Mon,b,\bar a)$ has finitely many components on $p(\Mon)$. By special monotonicity, the formula $\phi(x,y,\bar a)$ is a Boolean combination of $L_{\bar a}$-formulae: unary of the form $\theta_1(x)$ and $\theta_2(y)$, and $\mathcal S_{T_{\bar a}}$-monotone of the form $\psi_1(x,y)$ and $\psi_2(y,x)$. 
By the induction hypothesis each $\theta_1(\Mon)$ has finitely many convex components on $p(\Mon)$; the same holds for $\psi_1(\Mon,b)$ since it is a $<$-initial part. As for $\psi_2(b,\Mon)$, which is a final $<_{\vec E}$-part, it holds by Lemma \ref{remark convex set and new orders}. Then   $\phi(\Mon,b,\bar a)$, being a Boolean combination of these sets, has finitely many components on $p(\Mon)$.  
\end{proof}

\setcounter{TheoremI}{0}
\begin{TheoremI}\label{theorem 1}
Let $T$ be a complete first-order theory with infinite models.
\begin{enumerate}[(i)]
\item $T$ is weakly quasi-o-minimal  if and only if it is weakly monotone.
\item Weak monotonicity (weak quasi-o-minimality) of $T$ does not depend on the choice of a definable linear order: if $T$ is weakly monotone with respect to one definable linear order, then it is so with respect to any other.  
\end{enumerate}
\end{TheoremI}
\begin{proof}
Part (i) follows by the equivalence of conditions (1) and (3) in  Proposition \ref{thm mon iff wqom}. To prove part (ii), suppose that $T$ is weakly quasi-o-minimal with respect to $<$  and let $\lhd$ be another definable linear order. By Theorem \ref{thm wqom definable order} there is a definable partition $\Mon= D_1\cup\ldots\cup D_n$   and orders  $<_{\vec E_1},\ldots,<_{\vec E_n}\in \mathcal W_T$ such that $\lhd$ agrees with $<_{\vec E_i}$ on $D_i$ for all $i=1,\ldots,n$.
To  prove that $T$ is weakly quasi-o-minimal with respect to $\lhd$ by Proposition \ref{prop_WQOM_characterization}(i) it suffices to show  that any $\Mon$-definable   set $D\subseteq \Mon$ has only finitely many $\lhd$-convex components on $p(\Mon)$, so let us fix such $D$ and $p$. The locus $p(\Mon)$ is contained in a single $D_k$. By Lemma \ref{Plem_change_of_order_by_equivalence}, $T$ is weakly quasi-o-minimal with respect to $<_{\vec E_k}$, so $D$ has only finitely many $<_{\vec E_k}$-convex components on $p(\Mon)$ by Proposition \ref{prop_WQOM_characterization}(i). Since $\lhd$ agrees with $<_{\vec E_k}$ on $D_k$ and $p(\Mon)\subseteq D_k$, we conclude that $D$ has finitely many $\lhd$-convex components on $p(\Mon)$, as desired.
\end{proof}

\section{Proof of Theorem 2}

 \begin{Proposition}\label{Prop W_Tis S_T}
Suppose that  $T$ has a definable order $<$. Then the $\mathcal W_T$-reduct and the $\mathcal S_T$-reduct coincide. The family $\mathcal W_T$ is  $T$-complete and the theory $T_{\mathcal W_T}$ eliminates quantifiers.    
\end{Proposition}
\begin{proof} For a set $\mathcal F$ of definable linear orders let $Def(\mathcal F)$ denote the collection of all definable sets in the  $\mathcal F$-reduct. We will prove $Def(\mathcal W_T)= Def(\mathcal S_T)$. We claim:
\begin{equation}Def(\mathcal S_{T_{\mathcal S_T}})=Def(\mathcal W_{T_{\mathcal S_T}})\subseteq Def(\mathcal W_T)\subseteq Def(\mathcal S_T).
\end{equation}
Since $T_{\mathcal S_T}$ is a reduct of $T$, we have $\mathcal W_{T_{\mathcal S_T}}\subseteq  \mathcal W_T \subseteq  \mathcal S_T$ which implies  the inclusions in (1).
To justify the equality, first note that $T_{\mathcal S_T}$ is weakly quasi-o-minimal  by Corollary \ref{Cor S_T reduct is wqom},  so $Def(\mathcal S_{T_{\mathcal S_T}})=Def(\mathcal W_{T_{\mathcal S_T}})$  holds by Proposition \ref{thm mon iff wqom}.

Let $<_{\vec E}\in\mathcal S_T$.  The order $<_{\vec E}$ is clearly a $(<_{\vec E},<_{\vec E})$-monotone relation, so by $T$-complete\-ness of $\mathcal S_T$ (Proposition \ref{prop general via special}) $x<_{\vec E} y$ is definable by a Boolean combination of unary and $\mathcal S_T$-monotone formulae of the form $\psi(x,y)$; this implies  that $<_{\vec E}$ is definable in the $\mathcal S_T$-reduct. If $\vec E=(E_1,\ldots,E_{n})$, then we easily see that each $E_i$ is definable in terms of $<,<_{E_1},<_{(E_1,E_2)},\dots,<_{(E_1,E_2,\dots,E_{n})}$, so it is also definable in the  $\mathcal S_T$-reduct and  $<_{\vec E}\in \mathcal S_{T_{\mathcal S_T}}$. Therefore $\mathcal S_T\subseteq \mathcal S_{T_{\mathcal S_T}}$  and  $Def(\mathcal S_T)\subseteq Def(\mathcal S_{T_{\mathcal S_T}})$. Combining with (1), we get $Def(\mathcal W_T)= Def(\mathcal S_T)$; the $\mathcal W_T$-reduct and the $\mathcal S_T$-reduct coincide.

\smallskip
Finally, we prove $T$-completeness of $\mathcal W_T$. Let $R$ be a definable $(<_{\vec E},<_{\vec F})$-monotone relation, where $<_{\vec E},<_{\vec F}\in\mathcal W_T$.
Then $T$-completeness of $\mathcal S_T$ implies  that $R(x,y)$ is defined by a Boolean combination of unary and $\mathcal S_T$-monotone formulae of the form $\psi(x,y)$. Each of these formulae is named in the language of the theory $T_{\mathcal S_T}$ which is weakly quasi-o-minimal. By the equivalence of conditions (1) and (2) from Proposition  \ref{thm mon iff wqom} (applied to $T_{\mathcal S_T}$) each $\psi(x,y)$ is equivalent to a Boolean combination of unary and $\mathcal W_{T_{\mathcal S_T}}$-monotone formulae of the form $\theta(x,y)$. Since each $\mathcal W_{T_{\mathcal S_T}}$-monotone formula is also $\mathcal W_T$-monotone,  the formula $R(x,y)$ is equivalent to a Boolean combination of unary and $\mathcal W_T$-monotone formulae of the form $\theta(x,y)$; \ $\mathcal W_T$ is $T$-complete.

The last statement now follows from Proposition \ref{prop general simon}.
\end{proof}

\begin{TheoremI}\label{theorem 2}
Let $T$ be a weakly monotone theory  with a definable linear order $<$. 
\begin{enumerate}[(i)]
 \item  Every definable subset of $\Mon^2$ is defined by a Boolean combination of unary formulae and binary weakly monotone formulae.
 \item The theory of the full binary reduct  eliminates quantifiers. 
 \item If $\prec$ is a definable linear order on $\Mon$, then there are a  definable partition $\Mon= D_1\cup\ldots\cup D_n$   and orders  $<_{\vec E_1},\ldots,<_{\vec E_n}\in \mathcal W_T$ such that 
 $\prec$ agrees with $<_{\vec E_i}$ on $D_i$ for all $i=1,\ldots,n$.
\item (Weak monotonicity theorem)
For every definable set $D\subseteq \Mon$ and definable function $f:D\to\overline{\Mon}$ 
there exist  a definable partition $D=D_1\cup\dots\cup D_n$ and orders $<_{\vec E_1},\dots,<_{\vec E_n}\in\mathcal W_T$ such that each restriction $f\upharpoonright D_i$, viewed as a function from $(D_i,<_{\vec E_i})$ into $(\overline\Mon,<)$, is increasing for all $i=1,\ldots,n$.
\end{enumerate}
  \end{TheoremI}
\begin{proof}
(i) follows from Proposition \ref{thm mon iff wqom}(2).

\smallskip
(ii) Since $T$ is weakly monotone the full binary reduct and the $\mathcal W_T$-reduct coincide. The $\mathcal W_T$-reduct eliminates quantifiers by Proposition \ref{Prop W_Tis S_T}; so does the full binary reduct.

\smallskip
(iii) is proved in Theorem \ref{thm wqom definable order} and (iv) follows easily from Theorem \ref{thm wqom definable order} after noticing that $f(x)\leqslant f(y)$ defines a total preorder on $\Mon$. 
\end{proof}

\section{Concluding remarks}

First we note one more interesting consequence of weak monotonicity. Ramakrishnan and Steinhorn in \cite{RS} proved that a quasi-decisive, totally ordered $\aleph_0$-saturated structure $\mathcal M=(M,<,\dots)$, extends partial orders: every definable partial order on a definable subset of $M^n$ extends definably to a total order; $\mathcal M$ is quasi-decisive if:
for all $L_M$-definable subsets $B,C\subseteq M$ and types $p\in S_1(T)$ having realizations in $M$ coinitial  in $B$ ($B\cap p(M)$ has no strict lower bound in $B$),  there is an initial part of $B\cap p(M)$  that is either contained in or disjoint from $C$.

They showed that weakly o-minimal  and quasi-o-minimal structures are quasi-decisive. It is easy to see that weakly monotone structures are quasi-decisive, too. 

\begin{Corollary} 
$\aleph_0$-saturated weakly monotone structures extend partial orders. 
\end{Corollary}
\begin{proof}
Suppose that $\mathcal M=(M,<,\ldots)$ is $\aleph_0$-saturated and weakly monotone. It suffices to prove that $\mathcal M$ is quasi-decisive. Fix $L_M$-definable sets $B,C\subseteq M$ and $p\in S_1(T)$ having realizations in $M$ coinitial in $B$ (in fact, we only need that $B$ meets $p(M)$).
By Proposition \ref{prop_WQOM_characterization}(i), $B$ has finitely many convex components on $p(M)$; denote by $B_0$ the leftmost convex component of $B$ on $p(M)$. Also, $C$ and its complement determine a finite convex partition of $p(M)$, hence of its convex subset $B_0$ too. Let $B_0'\subseteq B_0$ be the leftmost member of this partition. Clearly $B_0'$ is an initial part of $B_0$, hence of $B\cap p(M)$ as well, but also $B_0'$ is either contained in or disjoint from $C$. Hence  $\mathcal M$ is quasi-decisive.
\end{proof}

We finish by stating some questions. The first is motivated by Corollary \ref{Corollary vece manje}, where we proved that the $\{<,>\}$-monotone reduct eliminates quantifiers. Call a theory $T$ {\it monotone with respect to $<$} if  for all (finite) sets $A\subset\Mon$  the full  binary reduct coincides with the $\{<,>\}$-reduct (both of them viewed as $T_A$-reducts); $T$ is {\it monotone} if it so with respect to some definable linear order. It sounds reasonable that monotone theories are close to quasi-o-minimal ones.

\begin{Question}\label{Question 8 2}
What would be an alternative description of monotone theories?
\end{Question} 

The second question is motivated by Remark \ref{rmk naming parameters and qwom}. There we noticed that weak monotonicity is preserved under naming  parameters, but is not necessarily preserved by dismissing them. The example for the latter was a dense linear order with an equivalence relation with two dense classes. Each class is $acl^{eq}(\emptyset)$-definable and undefinable; after naming a single element, each class becomes definable, and the theory becomes weakly monotone. 

\begin{Question}
If $<$ is a definable linear order, $A\subseteq \Mon$, $T_A$ is weakly monotone, and 
$acl^{eq}(\emptyset)=dcl^{eq}(\emptyset)$, must $T$ be weakly monotone, too?
\end{Question}

In this paper we gave a description of definable linear orders on $\Mon$ in weakly monotone theories. 
Onshuus and Steinhorn in \cite{OS} proved that definable linear orders on $\Mon^n$  definably piecewise-embed into lexicographic orders if the underlying theory is o-minimal with elimination of imaginaries.
This property is unlikely shared by general weakly o-minimal theories, but maybe it is so by binary ones; recall that a complete theory $T$ is {\it binary} if every formula is equivalent modulo $T$ to a Boolean combination of formulae in at most two free variables.
\begin{Question}
Is there a similar description, involving lexicographic orders and orders from $\mathcal W_T$ say, of linear orders that are interpretable in a binary weakly monotone theory?   
\end{Question}
Such a description would be quite interesting even in the case of colored orders.


\end{document}